\documentclass[twoside,12pt,reqno]{amsart}
\usepackage{bbm}
\usepackage{amssymb}
\usepackage{fullpage}

\newtheorem{Theorem}{Theorem}[section]
\newtheorem{Lemma}[Theorem]{Lemma}
\newtheorem{Corollary}[Theorem]{Corollary}

\theoremstyle{remark}
\newtheorem{Remark}[Theorem]{Remark}

\numberwithin{equation}{section}

\def\bar{\overline}
\def\sub{\subseteq}
\def\iso{\cong}
\def\into{\hookrightarrow}
\def\onto{\twoheadrightarrow}
\def\isoto{\overset{\sim}{\longrightarrow}}
\def\epsilon{\varepsilon}

\def\ad{\operatorname{ad}}
\def\Ad{\operatorname{Ad}}
\def\cdet{\operatorname{cdet}}
\def\col{\operatorname{col}}
\def\diag{\operatorname{diag}}
\def\gr{\operatorname{gr}}
\def\HC{{\operatorname{HC}}}
\def\Lie{\operatorname{Lie}}

\def\pr{\operatorname{pr}}
\def\row{\operatorname{row}}
\def\sspan{\operatorname{span}}
\def\sgn{\operatorname{sgn}}
\def\Tab{\operatorname{Tab}}
\def\tw{{\operatorname{tw}}}

\def\cns{\c_n(\sigma)}
\def\cnls{\c_{n,l}(\sigma)}
\def\Yns{Y_n(\sigma)}
\def\Ynps{Y_n^{[p]}(\sigma)}
\def\Ynls{Y_{n,l}(\sigma)}
\def\Ynlps{Y_{n,l}^{[p]}(\sigma)}

\def\C{{\mathbb C}}
\def\F{{\mathbb F}}
\def\K{{\mathbb K}}
\def\Z{{\mathbb Z}}

\def\GL{\mathrm{GL}}

\def\kk{{\mathbbm k}}

\def\ba{\text{\boldmath$a$}}
\def\bd{\text{\boldmath$d$}}
\def\bm{\text{\boldmath$m$}}
\def\bp{\text{\boldmath$p$}}
\def\bq{\text{\boldmath$q$}}
\def\bs{\text{\boldmath$s$}}
\def\bt{\text{\boldmath$t$}}
\def\bu{\text{\boldmath$u$}}
\def\bv{\text{\boldmath$v$}}
\def\bw{\text{\boldmath$w$}}

\def\bD{\text{\boldmath$D$}}
\def\bI{\text{\boldmath$I$}}
\def\bJ{\text{\boldmath$J$}}

\def\bi{\text{\boldmath$i$}}
\def\bj{\text{\boldmath$j$}}

\def\b{\mathfrak b}
\def\c{\mathfrak c}
\def\g{{\mathfrak g}}
\def\gl{\mathfrak{gl}}
\def\h{\mathfrak h}
\def\i{\mathfrak i}
\def\j{\mathfrak j}
\def\k{\mathfrak k}
\def\m{\mathfrak m}
\def\p{\mathfrak p}
\def\t{\mathfrak t}
\def\v{\mathfrak{v}}
\def\Sf{\mathfrak{S}}

\def\cF{\mathcal{F}}

\def\cM{\mathcal{M}}

\def\tc{\tilde{c}}
\def\te{\tilde{e}}
\def\tB{\widetilde{B}}

\def\dB{\dot{B}}
\def\dC{\dot{C}}
\def\dD{\dot{D}}
\def\dE{\dot{E}}
\def\dF{\dot{F}}
\def\dT{\dot{T}}
\def\dZ{\dot{Z}}

\newcommand{\arxiv}[1]{{\tt arXiv:#1}}

\title{\boldmath Restricted shifted Yangians and restricted finite $W$-algebras}
\author{Simon M.~Goodwin and Lewis Topley}
\address{School of Mathematics,
University of Birmingham,
Birmingham, B15 2TT,
UK}
\email{s.m.goodwin@bham.ac.uk}
\address{School of Mathematics, Statistics and Actuarial Science, University of Kent,
Canterbury, Kent CT2 7FS, UK}
\email{L.Topley@kent.ac.uk}


\begin{document}

\begin{abstract}
We study the truncated shifted Yangian $\Ynls$ over an algebraically
closed field $\kk$ of characteristic $p >0$, which is known to be
isomorphic to the finite $W$-algebra $U(\g,e)$
associated to a corresponding nilpotent element $e \in \g = \gl_N(\kk)$.
We obtain an explicit description of the centre of $\Ynls$,
showing that it is generated by its Harish-Chandra centre and its $p$-centre.
We define $Y_{n,l}^{[p]}(\sigma)$ to be the quotient
of $\Ynls$ by the ideal generated by the kernel of trivial character of its $p$-centre.
Our main theorem states that $Y_{n,l}^{[p]}(\sigma)$ is isomorphic to the restricted finite $W$-algebra $U^{[p]}(\g,e)$.
As a consequence we obtain an explicit presentation of this restricted $W$-algebra.
\end{abstract}

\maketitle

\section{Introduction}

Let $G$ be a reductive
algebraic group over an algebraically closed field $\kk$ of characteristic $p>0$, with Lie algebra $\g = \Lie G$.
The centre of $U(\g)$
admits a large {\em $p$-centre} $Z_p(\g)$ which is $G$-equivariantly isomorphic to the
coordinate ring of (the Frobenius twist of) $\g^*$.
For $\chi \in \g^*$
the reduced enveloping algebra $U_\chi(\g)$, is defined to be the
quotient of $U(\g)$ by the ideal generated by the maximal ideal of $Z_p(\g)$ corresponding to $\chi$.
The most important aspects of the representation theory of $\g$ are understood by studying $U_\chi(\g)$-modules,
and the early work of Kac--Weisfeiler, in \cite{KW}, shows that it suffices to consider the case $\chi$ nilpotent,
meaning $\chi$ identifies with a nilpotent element $e \in \g$ under some choice of $G$-equivariant isomorphism
$\g \cong \g^*$ (we assume the standard hypotheses).  We refer to \cite{JaLA} for a survey of this theory up to 2004,
and also to \cite{BM} for major developments based on deep connections with the geometry of Springer fibres.
In \cite{PrST} Premet made a significant breakthrough: he showed that any such
$U_\chi(\g)$ is Morita equivalent to a certain algebra $U^{[p]}(\g,e)$, now known as the
{\em restricted finite $W$-algebra}.

In this paper, we consider the case $G = \GL_N(\kk)$, so that $\g = \gl_N(\kk)$.
Our main theorem provides an explicit
presentation for the restricted finite $W$-algebra $U^{[p]}(\g,e)$.  This is achieved by exhibiting an
isomorphism with a restricted version of a truncated shifted Yangian, as stated in Theorem~\ref{T:restricted} below.

Before we proceed, we recall some relevant history. In \cite[Section~4]{PrST} Premet constructed
finite $W$-algebras over fields of characteristic zero, and since then these algebras have found many deep
applications to classical problems surrounding the
representations of complex semisimple Lie algebras; see \cite{PrMR} and \cite{Lo} for surveys
on this theory.

In \cite{BKshift}, Brundan--Kleshchev made a breakthrough by
providing a presentation of the complex finite $W$-algebra for the case $\g = \gl_N(\C)$ by
defining an explicit isomorphism with a certain quotient of a shifted Yangian.
This allowed them to make an extensive study of the representation theory of these finite $W$-algebras in
\cite{BKrep}.

Building on Premet's seminal work using the method of modular reduction of finite $W$-algebras, first
considered in \cite{PrPI} and exploited further in \cite{PrCQ},
the authors developed a direct approach to theory of finite $W$-algebras $U(\g,e)$ over $\kk$ in \cite{GTmod}.
Very briefly, for a choice of nilpotent $e\in \g$ corresponding to $\chi \in \g^*$,
the algebra $U(\g,e)$ is a filtered deformation of a good transverse slice
$\chi + \check\v$
to the
coadjoint orbit $G \cdot \chi$.
Further, $U(\g,e)$ admits a natural $p$-centre $Z_p(\g,e)$ isomorphic to
the coordinate algebra of (the Frobenius twist of) $\chi + \check\v$.
Then the restricted $W$-algebra $U^{[p]}(\g,e)$ is the quotient of $U(\g,e)$ by the ideal generated
by the ideal of $Z_p(\g,e)$
corresponding to $\chi$.

In joint work with Brundan \cite{BT} the second author developed the theory of shifted Yangians
$\Yns$ over $\kk$. One of the key features which differs from characteristic zero is the existence
of a large central subalgebra $Z_p(\Yns)$, called the {\em $p$-centre}, which is
constructed using some very natural power series formulas.

In subsequent work \cite{GTmin}, the authors showed that Brundan--Kleshchev's isomorphism
descends to positive characteristic. To explain this, we require a little notation, and from now
on we take $\g = \gl_N(\kk)$.  To each nilpotent element $e \in \g$ with Jordan type $\bp = (p_1 \le \cdots \le p_n)$,
we may associate a choice of shift matrix $\sigma = (s_{i,j})_{1\leq i,j\leq n}$, and thus a shifted Yangian
$\Yns$, which is a subalgebra of the Yangian $Y_n$.  The beautiful formulas introduced in \cite{BKshift} lead to a surjective
algebra homomorphism $\tilde\phi : \Yns \to U(\g,e)$.  Unsurprisingly the kernel of $\tilde \phi$ has the
same description as in characteristic zero, and so there is a natural isomorphism
\begin{equation*}
\phi : \Ynls \isoto U(\g,e),
\end{equation*}
where $\Ynls$ is the truncated shifted Yangian of level $l$, first
defined over the complex numbers in \cite[Section~6]{BKshift}.

Making use of the explicit presentation of $U(\g,e)$ obtained through the isomorphism $\phi$, it was proved in
\cite{GTmin} that every $U_\chi(\g)$-module of minimal dimension is parabolically induced.
This result is a modular analogue of M{\oe}glin's famous theorem on completely prime primitive ideals, see \cite{Moe},
and some of our methods adapt those in the proof given by Brundan in \cite{BrM}.

In this paper we define the $p$-centre $Z_p(\Ynls)$ of $\Ynls$ to be the image of $Z_p(\Yns)$ under the natural map $\Yns \onto \Ynls$,
and this leads to a {\em restricted truncated shifted Yangian}
$\Ynlps$ by taking the quotient of $\Ynls$ by the ideal generated by the natural generators of $Z_p(\Ynls)$.

We emphasise here that the origin of $Z_p(\Yns)$ is totally distinct from the construction of $Z_p(\g,e)$.
Nevertheless, our main theorem states that the isomorphism $\phi$ factors through the restricted
quotients.

\begin{Theorem} \label{T:restricted}
The isomorphism $\phi : \Ynls \isoto U(\g,e)$ factors to an isomorphism
$$
\phi^{[p]} : \Ynlps \isoto U^{[p]}(\g,e).
$$
\end{Theorem}

Since $\Ynlps$ is defined by generators and relations, the above theorem provides an explicit presentation for $U^{[p]}(\g,e)$.

The main ingredients of the proof are a detailed study of the centres of $\Ynls$ and $U(\g,e)$ together
with an analysis of highest weight modules for both algebras.  We emphasise that Theorems \ref{T:centreY}
and \ref{T:centreW} are significant results in their own right, describing the structures of the centres of $\Ynls$ and $U(\g,e)$
explicitly.  Furthermore, we expect the development of highest weight modules in Section~\ref{S:hwY} will
play an important role in future work.

Below we give an outline of the paper, in which
we point out the most important steps.

In Section \ref{S:prelims}, we recall some relevant preliminaries, and introduce the combinatorial
notation that we require.
There are new results in \S\ref{ss:Znilpcent}, where we consider the centre $Z(\g^e)$ of the universal enveloping
algebra of the centralizer of $e$.  In particular, we use \cite{BB} to give precise formulas for the generators of $Z(\g^e)$,
sharpening the main results of \cite{To}.  Also in \S\ref{ss:truncations}, we observe that $\g^e$ is isomorphic
to a truncated shifted current Lie algebra, which is useful later in this paper.

In Section \ref{S:YandW},
we recall the structural features of the shifted Yangian $\Yns$ and the finite $W$-algebra $U(\g,e)$,
drawing on \cite{BT}, \cite{GTmod} and \cite{GTmin}.
The key tools introduced here are the various filtrations on these algebras, and a precise description of their associated graded algebras.
We also recall the definition of the map $\tilde \phi$ lying at the core of our main theorem.
In \S\ref{ss:truncate} we introduce the truncation $\Ynls$ at level $l$, and use the shifted current algebra to simplify the
proof of the PBW theorem for $\Ynls$, see Theorem~\ref{T:isoandPBW}.
The main benefit of this slight simplification is that we may then apply the same argument to the integral forms of the Yangian and truncated shifted Yangian $Y_n^\Z(\sigma)$ and $Y_{n,l}^\Z(\sigma)$.
These integral forms, introduced in \S\ref{ss:Zforms}, are useful tools in some of our later proofs as they allow us to reduce modulo $p$ certain formulas from the characteristic zero case, see Corollary~\ref{C:redmodp}.
We expect these forms to find some independent interest, beyond the purposes of the present article.

Section~\ref{S:cent} is devoted to describing the centres of $\Ynls$ and $U(\g,e)$.
Our results are perfect analogues of Veldkamp's classical description of the centre $Z(\g)$
of $U(\g)$; see for example \cite[Theorem~3.5]{BG} and the references there.
We give definitions of the Harish-Chandra centres of $\Ynls$
and $U(\g,e)$; these are denoted by $Z_\HC(\Ynls)$ and $Z_\HC(\g,e)$,
and they are defined so that they ``lift'' the centre in characteristic zero.
The $p$-centres $Z_p(\Ynls)$ and $Z_p(U(\g,e))$ of $\Ynls$ and $U(\g,e)$ are also introduced here.
In Theorem~\ref{T:centreY} we give a detailed description of the centre of $\Ynls$, in particular
showing that is generated by
$Z_p(\Ynls)$ and $Z_\HC(\Ynls)$.
The next significant result is Theorem~\ref{T:centreW} in which we deduce
an analogous result for the centre $Z(\g,e)$ of $U(\g,e)$.
We mention that in recent work, Shu--Zeng have stated a more general result about the centre of modular finite $W$-algebras associated to arbitrary connected reductive groups, under
certain hypotheses, see \cite[Theorem 1]{SZ}. The more detailed
description we give here
is a necessary step in the proof of our main theorem,
and will play an role in future work.
A precise description of a set of generators for
$Z_p(\g,e)$ is given in \S\ref{ss:centreW}, and this is important in the sequel.
We also draw attention to Corollary~\ref{C:HClineup} which shows that $\phi$ preserves the
Harish-Chandra centres.
In \S\ref{ss:restrictedshifted} and \S\ref{ss:restrictedW} we discuss the
restricted quotients $\Ynlps$ and $U^{[p]}(\g,e)$ and their PBW bases.

In Section~\ref{S:hwY} we develop some highest weight theory for $\Ynls$ and
study the action of $U(\g,e)$ on highest weight modules through the Miura map.
One of the key ingredients of this theory is the use of a certain torus acting by automorphisms on both algebras, which is explained in detail in \S\ref{ss:torusactions}.
The key results after that are Lemmas~\ref{L:Birpaction} and \ref{L:DonUh}(c) which describe how the generators
of the $p$-centres $Z_p(\Ynls)$ and $Z_p(\g,e)$ act on highest weight modules.
Other important results for us are Corollaries~\ref{C:ZY0inj} and \ref{C:ZU0inj},
which concern analogues of Harish-Chandra homomorphisms for $\Ynls$ and $U(\g,e)$.

Finally, in Section~\ref{S:mainproof}, we combine our results to observe that the natural generators of $\phi(Z_p(\Ynls))$ act on highest weight vectors in precisely the same manner as the generators for $Z_p(\g,e)$. Using the Harish-Chandra restriction homomorphisms in $U(\g,e)$ we deduce that the ideal of $\Ynls$ generated by the kernel of the trivial character of $Z_p(\Ynls)$ is mapped to 
the ideal of $U(\g,e)$ generated by the kernel of the trivial character
of $Z_p(\g,e)$, and the main theorem follows quickly.

We remark that our proof does not show that $\phi : Z_p(\Ynls) \to Z_p(\g,e)$, and so it remains an interesting open problem to decide if these centres really do line up.

\subsection*{Acknowledgments}
The first author is supported by EPSRC grant EP/R018952/1, and the second author is
supported by EPSRC grant EP/N034449/1. We thank R.~Tange
for providing the idea for part of the proof of Corollary \ref{C:ZY0inj}.  We are also
grateful to J.~Brundan for helpful comments.

\section{Preliminaries and recollection} \label{S:prelims}
Throughout this paper, let $p \in \Z_{\ge 1}$ be a prime number, let $\F_p$ be the field of $p$ elements
and let $\kk$ be an algebraically closed field of characteristic $p$.

\subsection{A useful identity}
We require a standard identity in the polynomial ring $\kk[t]$ for the proof of Lemma~\ref{L:Birpaction},
and we recall it here.
Each $x \in \F_p$ satisfies $x^p - x = 0$, so for an indeterminate $t$, we
deduce that
\begin{equation} \label{e:Fppoly}
\prod_{j=0}^{p-1} (t-j) = t^p - t
\end{equation}
in $\F_p[t]$.  More generally, for any $a \in \kk$, we have the following equality in in $\kk[t]$
\begin{equation} \label{e:Fppolya}
\prod_{j=0}^{p-1} (t-a-j) = (t-a)^p - t-a = t^p - t  - (a^p - a).
\end{equation}

Observe that for $1 \le r \le p$ the coefficient of $t^{p-r}$ in the left hand side of \eqref{e:Fppoly} is $(-1)^{r} e_r(0,1,\dots,p-1)$, where $e_r(t_1,\dots,t_p)$ denotes the $r$th elementary symmetric polynomial in indeterminates $t_1,\dots,t_p$.
It follows that $e_r(0,1,\dots,p-1) = 0$ in $\F_p$ for $r = 1,\dots,p-2$; this gives a short alternative proof of \cite[Lemma 2.7]{BT}.

\subsection{Some standard results on algebras and modules}
We require a few elementary results from commutative and non-commutative algebra, which we state and prove for the reader's convenience.
The first lemma is well-known.
Let $A$ be a commutative $\kk$-algebra and $B, C \sub A$ subalgebras.
If $A$ is generated by $B \cup C$ then it follows that
there is a surjective homomorphism
$\phi : B \otimes_{B \cap C} C \onto A$.

\begin{Lemma}
\label{L:tensorprodpresentation}
Suppose that there exist elements $c_1,\dots,c_m \in C$ such that:
\begin{enumerate}
\item[(a)] the $B$-module generated by $c_1,\dots,c_m$ is free on $c_1,\dots,c_m$; and
\item[(b)] $C$ is generated by $c_1,\dots,c_m$ as a $B \cap C$-module.
\end{enumerate}
Then $\phi : B \otimes_{B \cap C} C \isoto A$ is an isomorphism.
\end{Lemma}

\begin{proof}
We just have to prove that $\phi$ is injective, so we let $y \in \ker \phi$.
It follows from (b) that
$y = \sum_{i=1}^m b_i \otimes c_i$, for some $b_i \in B$.  Then we have
$0 =\phi(y) = \sum_{i=1}^m b_ic_i$,
and this implies $b_i = 0$ for all $i = 1,\dots,m$ by (a), so that $y = 0$.
\end{proof}

The next result concerns free modules for a commutative $\kk$-algebra $A$.
It is well-known that
a surjective endomorphism of a finitely generated $A$-module is an isomorphism; this
can be proved using Nakayama's lemma, see for example \cite[Theorem 2.4]{Ma}.

\begin{Lemma} \label{L:freebasis}
Let $M$ be a free $A$-module of rank $n$ and let $m_1,\dots,m_n \in M$.
Suppose that $M$ is generated by $m_1,\dots,m_n$ as an $A$-module.  Then
$M$ is free on $m_1,\dots,m_n$.
\end{Lemma}

\begin{proof}
Let $x_1,\dots,x_n \in M$ be free generators of $M$ as an $A$-module.  Consider the endomorphism $\theta : M \to M$
defined by $\theta(x_i) = m_i$.  Since $M$ is generated by $m_1,\dots,m_n$, we have that $\theta$ is surjective, and
thus an isomorphism.  Hence, $M$ is free on $m_1,\dots,m_n$.
\end{proof}

The final result in this subsection
is required several times in the sequel, and included for convenience of reference.
Let $A$ be a non-negatively filtered (not necessarily commutative) $\kk$-algebra with filtered pieces $\cF_i A$
for $i \in \Z_{\ge 0}$.  Also let $M$ be a non-negatively filtered $A$-module with filtered pieces $\cF_i M$ for $i \in \Z_{\ge 0}$.
We write $\gr A$ for the associated graded algebra of $A$ and $\gr M$ for the associated graded module of $M$.
If $m \in \cF_i M$ then the notation $\gr_i m := m + \cF_{i-1} M\in \gr M$ will be used throughout the paper.
The following lemma can be proved with a standard filtration argument.

\begin{Lemma}
\label{L:filtfree}
Suppose that $\gr M$ is free as a graded $\gr A$-module with homogeneous basis $\{\gr_{d_i} m_i \mid i \in I\}$, where $I$ is some index
set, $d_i \in \Z_{\ge 0}$ and $m_i \in \cF_{d_i}M$.  Then $M$ is a free $A$-module with basis $\{m_i \mid i \in I\}$.
\end{Lemma}

\subsection{Algebraic groups and restricted Lie algebras} \label{ss:alggps}
We introduce some standard notation for algebraic groups and their Lie algebras,
which is used in the sequel.  Let $H$ be a linear algebraic group over $\kk$, and let $\h = \Lie H$ be the
Lie algebra of $H$.  We write $U(\h)$ for the universal enveloping algebra of $\h$, and
$Z(\h)$ for the centre of $U(\h)$.  We denote the $i$th filtered piece of $U(\h)$ in the standard PBW filtration
by $F_i U(\h)$.  The associated graded algebra $\gr U(\h)$ is identified with $S(\h)$, the symmetric algebra of $\h$.

The adjoint action of $H$ on $\h$ extends
to an action on $U(\h)$.
Also
$S(\h)$ has adjoint actions of $H$ and $\h$.  We use the
standard notation $(h,u) \mapsto \Ad(h)u$ and $(x,u) \mapsto
\ad(x)u$ for these actions, where $h \in H$, $x \in \h$, and $u \in U(\h)$ or $u \in S(\h)$.
For a closed subgroup $K$
of $H$ and $K$-stable subspace $A$ of $U(\h)$ or of $S(\h)$, we write $A^K$
for the invariants of $K$ in $A$ and $A^\k$ for the invariants of $\k$ in $A$.
Given $x \in \h$, we write $\h^x$ for the centralizer
of $x$ in $\h$,
and we write $H^x$ for the centralizer of $x$ in $H$.

We have that $\h$ is a restricted Lie algebra
and we write $x \mapsto x^{[p]}$ for the $p$-power map.  The $p$-centre of
$U(\h)$ is the subalgebra $Z_p(\h)$ of $Z(\h)$ generated by $\{x^p - x^{[p]} \mid x \in \h\}$.
There is a $H$-equivariant isomorphism $\xi = \xi_\h : S(\h)^{(1)} \to Z_p(\h)$,
determined by $\xi_\h(x) = x^p-x^{[p]}$ for $x \in \h$;
here $S(\h)^{(1)}$ denotes the Frobenius twist of $S(\h)$.

\subsection{Combinatorial notation} \label{ss:comb}

We require various pieces of combinatorial notation, which we set out below.

By a {\em composition} we simply mean a sequence $\bq = (q_1,q_2,\dots)$,
where $q_i \in \Z_{\ge 0}$ and only finitely many are nonzero.  When $\bq$
is a composition, $l \in \Z_{\ge 0}$ and $q_i = 0$ for all $i > l$, we
write $\bq = (q_1,\dots,q_l)$.  Given a composition $\bq$, we
define $|\bq| = \sum_{i \ge 1} q_i$ and say that $\bq$ is a composition
of $|\bq|$.  Also we define $\ell(\bq) = |\{i \in \Z_{\ge 1} \mid q_i > 0\}|$.
In this paper a composition $\bp$ is called a {\em partition}
if $0 < p_i \le p_{i+1}$ for all $1 \le i < \ell(\bp)$.
Given two compositions $\bm$ and $\bp$, we say that $\bm$ is a {\em subcomposition}
of $\bp$ if $m_i \le p_i$ for all $i \in \Z_{\ge 1}$, and in this case we write $\bm \sub \bp$.

Let $n \in \Z_{\ge 0}$.  By a {\em shift matrix} of size $n$ we mean
a $n \times n$ matrix $\sigma = (s_{i,j})$ with entries in $\Z_{\ge 0}$
such that $s_{i,j} = s_{i,k} + s_{k,j}$ whenever $i \le k \le j$, or $i \ge k \ge j$.
We note that this implies that $s_{i,i} = 0$ for all $i$, and that $\sigma$ is completely
determined by the entries $s_{i,i+1}$ and $s_{i+1,i}$ for $i=1,\dots,n-1$.

Let $N \in \Z_{\ge 0}$ and let $\bq = (q_1,\dots,q_l)$ be a composition of $N$ such that
for some $j$ we have $0< q_1 \le \dots \le q_j \ge \dots \ge q_l > 0$, and let $n := q_j = \max_i q_i$.
We define the pyramid $\pi = \pi(\bq)$
to be the diagram made up of $N$ boxes stacked in columns of heights $q_1,\dots,q_l$.
We let $\bp = \bp(\bq)$ be the partition of $N$ giving the row lengths of $\pi$ from top to bottom;
note that the number $p_n = l$ will often be referred to as {\em the level}.
The boxes in $\pi$ are labelled with $1,\dots,N$ along rows from left to right and from top to bottom.
The columns of $\pi$ are labelled $1,2,\dots,l$ from left to right and the rows are
labelled $1,2,\dots,n$ from top to bottom.
The box in $\pi$ containing $i$ is referred to as the $i$th box, and we write
$\row(i)$ and $\col(i)$ for the row and column of the $i$th box respectively.
We define the shift matrix $\sigma = \sigma(\bq)$ from $\pi$ by setting $s_{j,i}$ to be
the left indentation of the $i$th row of $\pi$ relative to the $j$th row, and $s_{i,j}$
to be the right indentation of the $i$th row of $\pi$ relative to the $j$th row, for $1 \le i \le j \le n$.

As an example we consider $\bq = (1,3,3,2,1)$.  The pyramid is
\begin{equation*}
\pi = \begin{array}{c}
\begin{picture}(60,36)
\put(0,0){\line(1,0){60}}
\put(0,12){\line(1,0){60}} \put(12,24){\line(1,0){36}}
\put(12,36){\line(1,0){24}}
\put(0,0){\line(0,1){12}} \put(12,0){\line(0,1){36}}
\put(24,0){\line(0,1){36}} \put(36,0){\line(0,1){36}}
\put(48,0){\line(0,1){24}}
\put(60,0){\line(0,1){12}}
\put(15,26.5){\hbox{1}}
\put(27,26.5){\hbox{2}}
\put(15,14.5){\hbox{3}}
\put(27,14.5){\hbox{4}}
\put(39,14.5){\hbox{5}}
\put(3,2.5){\hbox{6}}
\put(15,2.5){\hbox{7}}
\put(27,2.5){\hbox{8}}
\put(39,2.5){\hbox{9}}
\put(47.5,2.5){\hbox{10}}
\end{picture}
\end{array}.
\end{equation*}
Then we obtain the partition $\bp = (2,3,5)$, and the shift matrix
\begin{equation*}
\label{e:shiftexmp}
\sigma = \left(\begin{array}{ccc} 0 & 1 & 2  \\ 0 & 0 & 1 \\
1 & 1 & 0 \end{array}\right).
\end{equation*}

Evidently the data encoded in the composition $\bq$ is equivalent to the data given by the pyramid $\pi$.
We have explained how to construct a shift matrix and a level $(\sigma,l)$ from a pyramid.  To complete
the picture we observe that we can build the pyramid $\pi$ from knowledge of $(\sigma,l)$, by starting with
a bottom row of length $l$, and indenting the higher rows according to $\sigma$.
The partition $\bp$ can be explicitly recovered from $(\sigma, l)$ by the rule
\begin{equation}
\label{e:sigmalgivesp}
p_i := l - s_{i,n} - s_{n,i}
\end{equation}
Therefore, the combinatorial data $\bq$, $\pi$ and $(\sigma, l)$ are all equivalent.

Let $\pi = \pi(\bq)$ be a pyramid.
A {\em $\pi$-tableau} is a diagram obtained by filling the boxes of $\pi$ with elements of $\kk$.
The set of all tableaux of shape $\pi$ is denoted $\Tab_\kk(\pi)$.
For $A \in \Tab_\kk(\pi)$, we write $a_i$ for the entry in the $i$th box of $A$; alternatively
we sometimes write $a_{i,1},\dots,a_{i,p_i}$ for
the entries in the $i$th row of $A$ from left to right.
Two $\pi$-tableaux are called {\em row-equivalent} if one can be obtained from the other
by permuting the entries in the rows.

\subsection{Nilpotent elements in $\gl_N(\kk)$ and their centralizers}
\label{ss:nilpcent}

Let $\pi$ be a pyramid with partition $\bp = (p_1,\dots,p_n)$,
such that $|\bp| = N$. Let $G = \GL_N(\kk)$,
so $\g = \gl_N(\kk)$, which is a restricted Lie algebra with $p$-power map
given by the $p$th matrix power.
We write $\{e_{i,j} \mid 1 \le i,j \le N\}$ for the standard basis of $\g$ consisting
of matrix units.

The pyramid $\pi$ is used to determine the nilpotent element
\begin{equation} \label{e:enilp}
e := \sum_{\substack{\row(i) = \row(j) \\ \col(i) = \col(j)-1}} e_{i,j} \in \g,
\end{equation}
which has Jordan type $\bp$.
Note that $e$ depends only on $\bp$ and not the choice of pyramid $\pi$.

The centraliser $\g^e$ of $e$ in $\g$ has a basis
\begin{equation}
\label{e:centraliserbasis}
\{c_{i,j}^{(r)} \mid 1\le i,j \le n, \ s_{i,j} \le r < s_{i,j} + p_{\min(i,j)}\}
\end{equation}
where
\begin{equation*}
c_{i,j}^{(r)} :=
\sum_{\substack{1\le h,k,\le N\\ \row(h) = i, \row(k) = j \\ \col(k) - \col(h) = r}} e_{h,k} \in \gl_N.
\end{equation*}
This is stated for example in \cite[Lemma~2.1]{GTmin}, although we warn the reader that the notation
used here and there differs by a shift by one in the superscripts.
In \cite[Lemma~2.1]{GTmin} it is also stated that the Lie brackets are given by
\begin{equation}
\label{e:centraliserrelations}
[c_{i,j}^{(r)}, c_{k, l}^{(s)}] = \delta_{j,k} c_{i,l}^{(r+s)} - \delta_{i,l} c_{k,j}^{(r+s)}.
\end{equation}
It is straightforward to see that the $p$-power map on $\g^e$ is given by
\begin{equation}
\label{e:centraliserpstructure}
(c_{i,j}^{(r)})^{[p]} = \delta_{i,j} c_{i,j}^{(rp)}.
\end{equation}
To make sense of these formulas we adopt the convention, here and throughout, that $c_{i,j}^{(r)} = 0$ when $r \ge s_{i,j} + p_{\min(i,j)}$.

We note here that the labelling of the basis of $\g^e$ given in \eqref{e:centraliserbasis}
does depend on the choice of pyramid $\pi$.  However, the elements in the basis only depends
on the partition $\bp$, and relabelling between different choices of pyramids just involves
shifting the superscripts.

\subsection{The centre of the enveloping algebra of the centralizer} \label{ss:Znilpcent}
We now go on to describe the centre $Z(\g^e)$ of $U(\g^e)$. Such a description was first obtained by the second author in \cite{To}, however we will need
a much more precise formulation of this result which is compatible
with the theory of Yangians.
As such we draw heavily on the description of $Z(\g^e)$ given by Brown--Brundan \cite[Main theorem]{BB} in characteristic zero.
In {\em loc.\ cit.}, the statement is given for the case that the pyramid $\pi$ is left justified, which is equivalent to
the condition $s_{i,j} = 0$ when $i > j$ on the shift matrix $\sigma$, i.e.\ that $\sigma$ is upper triangular. From there it is easy to deduce a description of $Z(\g^e)$ in terms of the basis of $\g^e$ corresponding to any pyramid, as this involves is a trivial change of notation.
For this reason {\em we assume that $\pi$ is left justified up to and including Lemma \ref{L:Ucentralisercentre}},
so that our notation is aligned with that of \cite{BB}.

We begin by stating some formulas for elements of
$U(\g^e)$ which appeared in \cite[(1.3)]{BB} over $\C$.
Define the elements
\begin{equation*}
\tc_{i,j}^{(r)} := c_{i,j}^{(r)} - \delta_{r,0} \delta_{i,j} (i-1) p_i \in U(\g^e)
\end{equation*}
over the indexing set $\{(i,j,r) \mid
1 \le i,j \le n,\, 0 \le r < p_{\min(i,j)}\}$.
Then define the sequence
\begin{equation}
\label{e:invariantdegree}
(d_1,\dots.,d_N) = (\underbrace{1,\dots.,1}_{p_n \text{ times}},
\underbrace{2,\dots,2}_{p_{n-1} \text{ times}},\dots,\underbrace{n,\dots,n}_{p_1 \text{ times}}).
\end{equation}
This sequence is known to give the total degrees of a set of homogeneous generators of $S(\g^e)^{G^e}$, as is explained in \cite[Section 3]{To}.

For a subcomposition $\bm$ of $\bp$ such that $\ell(\bm) = d_{|\bm|}$, with nonzero entries
$m_{i_1},\dots,m_{i_d}$, where $d = d_{|\bm|}$  we define the $\bm$-column determinant of $(\tc_{i,j}^{(r)})$ to be
\begin{equation} \label{e:cdetcij}
\cdet_\bm(\tilde c_{i,j}^{(r)}) = \sum_{w\in \Sf_d} \sgn(w) \tc_{i_{w1}, i_1}^{(m_{i_1}-1)} \cdots \tc_{i_{wd}, i_d}^{(m_{i_d}-1)},
\end{equation}
where $\Sf_d$ denotes the symmetric group of degree $d$.
It is shown in \cite[Lemma 3.8]{BB} that all $\tc_{i_{wj}, i_j}^{(m_{i_j}-1)}$ involved in the above definition
of $\cdet_\bm(\tilde c_{i,j}^{(r)})$ are defined, i.e.\ that
$s_{i_{wj},i_j} = p_{i_j} - p_{\min(i_{wj},i_j)} < m_{i_j} \le p_{i_j} = p_{i_wj} + s_{i_{wj}, i_j}$.

Finally for $s = 1,\dots,N$ we define
\begin{equation}
\label{e:definezr}
z_s := \sum_{\substack{\bm \sub \bp \\ |\bm| = s,\, \ell(\bm) = d_s}} \cdet_\bm(\tc_{i,j}^{(r)}).
\end{equation}

We move on to state Lemma \ref{L:Ucentralisercentre}, which can essentially be
deduced from \cite[Theorem~3]{To}.
As our statement is slightly different and more explicit, we include an outline of the proof.

\begin{Lemma}
\label{L:Ucentralisercentre}
$ $
\begin{enumerate}
\item[(a)] The elements $z_1,\dots,z_N$ are algebraically independent generators of
$U(\g^e)^{G^e}$;
\item[(b)] $Z(\g^e)$ is a free $Z_p(\g^e)$-module of rank $p^N$ with basis
$\{z_1^{k_1} \cdots z_N^{k_N} \mid 0 \le k_i < p\}$,
and $Z_{HC}(\g^e)$ is a free $Z_p(\g^e)^{G^e}$-module with the same basis.
\item[(c)] The multiplication map $Z_p(\g^e) \otimes_{Z_p(\g^e)^{G^e}} U(\g^e)^{G^e} \to Z(\g^e)$
is an isomorphism.
\end{enumerate}
\end{Lemma}

\begin{proof}
We begin the proof by briefly considering the situation when $\kk$ has characteristic 0.
In this case,  using the fact that $G^e$ is connected, we have that $Z(\g^e) = U(\g^e)^{G^e}$.
Since $\pi$ is assumed to be left justified, the statement (a) in characteristic 0 is precisely \cite[Main Theorem]{BB}.
Now a reduction modulo $p$ argument, identical to that given in the proof of \cite[Corollary~1]{To},
can be used to deduce that $z_s \in U(\g^e)^{G^e}$
for $\kk$ of characteristic $p$.

By the definition given in  \eqref{e:definezr} we have that $z_s \in F_{d_s} U(\g^e)$ in the PBW filtration, for all $s$, and
$$
\gr_{d_s} z_s = \sum_{\substack{\bm \sub \bp \\ |\bm| = s,\, \ell(\bm) = d_s}} \cdet_\bm(c_{i,j}^{(r)}) \in S(\g^e).
$$
In \cite[Theorem~9]{To} it was demonstrated that $\{\gr_{d_s} z_s \mid s=1,\dots,N\}$ are algebraically independent generators of $S(\g^e)^{G^e}$.
We should warn the reader that the notation in {\it loc. cit.\ }was different: the partition $\bp$ was denoted $\lambda$,
the element $c_{i,j}^{(r)}$ was denoted $\xi_j^{i, r + p_j - p_i}$, and the notation $x_s$ was used to
denote the element determined by the formula for $\gr_{d_s} z_s$ above.
Now standard filtration arguments show that $z_1,\dots,z_N$ are algebraically independent,
and generate $U(\g^e)^{G^e}$.  This completes the proof of (a).

Taking associated graded algebras we have $\gr Z_p(\g^e) = S(\g^e)^p$
and $\gr Z(\g^e) \sub S(\g^e)^{\g^e}$, however this inclusion is actually an equality thanks to the proof of \cite[Theorem~3]{To}.  It follows from \cite[Theorem~9]{To}
that $\{ (\gr_{d_1} z_1)^{k_1} \cdots (\gr_{d_N} z_N)^{k_N} \mid 0 \le k_i < p\}$  generates $S(\g^e)^{\g^e}$ as a $S(\g^e)^p$-module.
and also that $S(\g^e)^{\g^e}$ is free of rank $p^N$ over $S(\g^e)^p$.
Therefore, $\{ (\gr_{d_1} z_1)^{k_1} \cdots (\gr_{d_N} z_N)^{k_N} \mid 0 \le k_i < p\}$
is in fact a basis of $S(\g^e)^{\g^e}$ over $S(\g^e)^p$ by Lemma~\ref{L:freebasis}.
Now we can use that $(\gr_{d_1} z_1)^{k_1} \cdots (\gr_{d_N} z_N)^{k_N} = \gr_{k_1 d_1+\dots+k_N d_N} (z_1^{k_1} \cdots z_N^{k_N})$
for any choice of $k_1,\dots,k_N$ and apply Lemma~\ref{L:filtfree} to obtain the first assertion in (b).

Next we observe that $U(\g^e)^{G^e} \cap Z_p(\g^e) = Z_p(\g^e)^{G^e}$, and that
$\gr Z_p(\g^e)^{G^e} = (S(\g^e)^{G^e})^p$.  It is clear that $S(\g^e)^{G^e}$ is free as an $(S(\g^e)^{G^e})^p$-module
with basis $\{ \gr z_1^{k_1} \cdots \gr z_N^{k_N} \mid 0 \le k_i < p\}$.  Therefore, using Lemma~\ref{L:filtfree}, we deduce that
$U(\g)^{G^e}$ is free as a $Z_p(\g^e)^{G^e}$-module with basis $\{ z_1^{k_1} \cdots z_N^{k_N} \mid 0 \le k_i < p\}$ giving
the second assertion in (b).  Now
we can apply Lemma~\ref{L:tensorprodpresentation} to obtain (c).
\end{proof}

We consider the special case where $e=0$, i.e.\ when $\bp = (1,\dots,1)$.
Here we can be
more explicit about the generators of $U(\g)^G$ as we explain below, where we observe that these generators
arise from the Capelli identity.  These generators of $U(\g)^G$ are well-known in characteristic zero,
see for example \cite[\S3.8]{BKrep},
and we expect it is also known in positive characteristic,
so we just give a short justification for convenience.

Recall that the {\em column determinant} $\cdet(A)$
of a square $N \times N$-matrix $A = (a_{i,j})_{1\le i,j\le N}$ with coefficients in
an associative algebra is defined by
\begin{equation} \label{e:cdetdefn}
\cdet(A) = \sum_{w \in \Sf_N} \sgn(w) a_{w(1), 1}\cdots a_{w(N),N}.
\end{equation}
Let $u$ be a formal variable and consider the determinant
\begin{align}
\label{e:ZcentralUg}
Z^*(u) & = u^N + \sum_{r=1}^N Z^{(r)} u^{N-r} \nonumber \\
&:= \cdet
\left( \begin{array}{cccc} e_{1,1} + u &  e_{1,2}  & \cdots & e_{1,N}\\
e_{2, 1} & e_{2, N-1} + u - 1 & \cdots & e_{2, N} \\
\vdots & \vdots & \ddots & \vdots\\
e_{N, 1} & e_{N, 2} & \cdots  & e_{N, N} + u - N + 1 \\
\end{array} \right),
\end{align}
where the entries of the matrix are considered as elements of $U(\g)[u]$.

For $\bm \sub \bp$,
we let $z^0_\bm$ be the matrix formed
by the rows and columns indexed by the set of $i$ such that $m_i = 1$ of the matrix appearing in \eqref{e:ZcentralUg},
after replacing a diagonal entry $e_{i,i}+u-i+1$  by $e_{i,i}-i+1$.
Using the formula for calculating the column determinant in \eqref{e:cdetdefn}
we can get the decomposition
$$
Z^*(u) = \sum_{\bm \sub \bp} u^{N-|\bm|} \cdet z^0_\bm.
$$
Now we  observe that $\cdet z^0_\bm$ is equal to $\cdet_\bm(\tilde c_{i,j}^{(r)})$
as defined in \eqref{e:cdetcij}, noting that in the present case
where $\bp = (1,\dots,1)$, we have $\tilde c_{i,j}^{(0)} = e_{i,j} - \delta_{i,j} (i-1)$.

Putting this all together, we can deduce that $z_r$ as defined in \eqref{e:definezr} is equal to $Z^{(r)}$ as defined
in \eqref{e:ZcentralUg}.  Consequently, the statements in Lemma~\ref{L:Ucentralisercentre} hold for $U(\g)$
with $Z^{(r)}$ in place of $z_r$.

We return to the case of general $e$, and we record an important technical lemma characterising the $p$-centre of $U(\g^e)$,
which is crucial to our later arguments.  Before this is stated in Lemma~\ref{L:gepcent}, we need to
give some more notation.  We let $\bar \pi$ be the pyramid obtained from $\pi$ by adding
an extra row on the bottom with $p_n$ boxes.  Then let $\bar \g = \gl_{N+p_{n}}(\kk)$, and let $\bar e \in \bar \g$
be the nilpotent element corresponding to $\bar \pi$.  The centralizer $\bar \g^{\bar e}$ has basis given by
$\{c_{i,j}^{(r)} \mid 1\le i,j \le n+1, \ s_{i,j} \le r < s_{i,j} + p_{\min(i,j)}\}$, where we extend the notation
used in \eqref{e:centraliserbasis}, setting $p_{n+1} := p_n$.  Inspecting \eqref{e:centraliserrelations} and \eqref{e:centraliserpstructure} we see that $\g^e$
identifies naturally with a restricted subalgebra of $\bar \g^{\bar e}$.

\begin{Lemma} \label{L:gepcent}
$Z_p(\g^e) = U(\g^e) \cap Z(U(\bar \g^{\bar e}))$.
\end{Lemma}

\begin{proof}
Clearly we have $Z_p(\g^e) \sub U(\g^e) \cap Z(U(\bar \g^{\bar e}))$. Suppose
that this inclusion is strict and let $z \in U(\g^e) \cap Z(U(\bar \g^{\bar e})) \setminus
Z_p(\g^e)$ such that $z \in F_d U(\g^e)$ with $d$ as small as possible.  If
$\gr_d z = y^p \in S(\g^e)^p$ where $y \in S(\g^e)$, then $z - \xi_{\g^e}(y) \in U(\g^e) \cap Z(U(\bar \g^{\bar e})) \setminus
Z_p(\g^e)$ and $z - \xi_{\g^e}(y) \in F_{d-1} U(\g^e)$.
Thus we have that $\gr_d z \in S(\g^e) \cap S(\bar \g^{\bar e})^{\bar \g^{\bar e}} \setminus
S(\g^e)^p \ne \varnothing$.

Let $y \in S(\g^e) \cap S(\bar \g^{\bar e})^{\bar \g^e} \setminus
S(\g^e)^p$.
Since $S(\g^e)$ is a free $S(\g^e)^p$-module we may define $I = \{(i,j,r) \mid 1 \le i,j \le n,\ s_{i,j} \le r < s_{i,j} + p_{\min(i,j)}\}$, and write
$y = \sum_{m} f_m \prod_{(i,j,r) \in I} (c_{i,j}^{(r)})^{m(i,j,r)},$
for certain elements $f_m \in S(\g^e)^p$, where the sum is taken over all maps $m : I \to \{0,\dots,p-1\}$. Since $y \notin S(\g^e)^p$ there exists an $m_0 : I \to \{0,\dots,p-1\}$ and a tuple $(i_0,j_0,r_0) \in I$ such that $f_{m_0} \neq 0$ and $m_0(i_0,j_0,r_0) \neq 0$. Using \eqref{e:centraliserrelations} we can write
\begin{multline*}
\ad\left(c_{n+1,i_0}^{(s_{n+1, i_0})}\right) y = f_{m_0} m_0(i_0, j_0, r_0) c_{n+1, j_0}^{(s_{n+1, i_0} + r_0)} (c_{i_0, j_0}^{(r_0)})^{m_0(i_0,j_0,r_0)-1} \prod_{(i_0, j_0, r_0) \neq (i,j,r)} (c_{i,j}^{(r)})^{m_0(i,j,r)}\\
  + f_{m_0} (c_{i_0, j_0}^{(r_0)})^{m_0(i_0,j_0,r_0)} \ad\left(c_{n+1,i_0}^{(s_{n+1, i_0})}\right) \prod_{(i_0, j_0, r_0) \neq (i,j,r)} (c_{i,j}^{(r)})^{m_0(i,j,r)}\\
  + \sum_{m\neq m_0} f_m \ad\left(c_{n+1,i_0}^{(s_{n+1, i_0})}\right) \prod_{(i,j,r) \in I} (c_{i,j}^{(r)})^{m(i,j,r)}.
\end{multline*}
Since $r_0 \leq s_{i_0, j_0} + p_{\min(i_0, j_0)}$ and $s_{n+1, i_0} + s_{i_0, j_0} \leq s_{n+1, j_0}$ it follows that $s_{n+1, i_0} + r_0 \leq s_{n+1,j_0} + p_{\min(n+1, j_0)}$.
In particular, $c_{n+1, j_0}^{(s_{n+1, i_0} + r_0)} \neq 0$ and so the summand occurring in the first line of the above expression for $\ad(c_{n+1,i_0}^{(s_{n+1, i_0})}) y$ is non-zero. Now it remains to observe that
 the non-zero monomial summands occurring in the expressions $$\{\ad\left(c_{n+1,i_0}^{(s_{n+1, i_0})}\right)\prod_{(i,j,r) \in I} (c_{i,j}^{(r)})^{m(i,j,r)} \mid m : I \to \{0,1,\dots,p-1\}\}$$ are all distinct;
this follows readily from \eqref{e:centraliserrelations}.
We conclude that $\ad(c_{n+1,i_0}^{(s_{n+1, i_0})}) y \neq 0$ which contradicts the assumption $y \in S(\bar \g^{\bar e})^{\bar \g^{\bar e}}$.

This contradiction confirms that the inclusion $Z_p(\g^e) \sub U(\g^e) \cap Z(U(\bar \g^{\bar e}))$ is actually an equality.
\end{proof}

\subsection{The truncated shifted current Lie algebra}
\label{ss:truncations}
Let $n \in \Z_{\ge 0}$.
The {\em current Lie algebra of $\gl_n(\kk)$}
is the Lie algebra $\c_n := \gl_n(\kk) \otimes \kk[t]$. For $x \in \gl_n(\kk)$ and $f \in \kk[t]$
we abbreviate our notation by writing $xf$ for $x \otimes f \in \c_n$, and observe that $\c_n$ has a basis
\begin{equation*}
\{e_{i,j} t^r \mid 1 \le i,j \le n, \, \, \, r \ge 0\},
\end{equation*}
where we write $\{e_{i,j} \mid 1 \le i,j \le n\}$ for the
standard basis of matrix units in $\gl_n(\kk)$.
The commutator between elements in this basis is given by
\begin{equation} \label{e:currentcommutator}
[e_{i,j} t^r,e_{k,l} t^s] = (\delta_{j,k}e_{i,l}-\delta_{i,l}e_{k,j})t^{r+s}.
\end{equation}
We have that $\c_n$ is a restricted Lie algebra with the $p$-power map defined
by $(xf)^{[p]}:= x^{[p]}f^p$ for $x \in \gl_n(\kk)$ and $f \in \kk[t]$, where
$x^{[p]}$ denotes the $p$th matrix power of $x$, see for example \cite[Lemma~3.3]{BT}.
So in particular the $p$-power map is given on the basis of $\c_n$ by
\begin{equation} \label{e:currentppower}
(e_{i,j}t^r)^{[p]} = \delta_{i,j} e_{i,j} t^{pr}.
\end{equation}

Now let $\sigma = (s_{i,j})$ be any shift matrix of size $n$.
The {\em shifted current Lie algebra} is defined to be the
subspace $\c_n(\sigma)$ of $\c_n$ spanned by
\begin{equation} \label{e:csbasis}
\{e_{i,j} t^r \mid 1 \le i,j \le n , \, \, \, r \ge s_{i,j}\}.
\end{equation}
It is observed in \cite[Lemma~3.3]{BT} that $\c_n(\sigma)$ is a restricted Lie subalgebra of $\c_n$.

We fix an integer $l > s_{1,n} + s_{n,1}$
which we call {\em the level}, following the terminology of \S\ref{ss:comb}.
Then using \eqref{e:sigmalgivesp} we define the partition
$\bp = (p_1,\dots,p_n)$ from the data $(\sigma, l)$, and we let $N = \sum_{i=1}^n p_i$.
We define the {\em truncated shifted current Lie algebra}
$\cnls$ to be the quotient of $\cns$
by the
ideal $\i_{n,l}$ generated by $\{e_{1,1}t^r \mid r \ge p_1\}$.

We recall from \S\ref{ss:comb} that $(\sigma,l)$ determines a pyramid $\pi$, which we can
use to define $e \in \g = \gl_N(\kk)$ as in \eqref{e:enilp}. The next lemma shows that
the truncated current Lie algebra is isomorphic to the centralizer $\g^e$.  For the
statement we recall that a basis is given in \eqref{e:centraliserbasis}.

\begin{Lemma}
\label{L:truncatedshiftedbasis}
$ $
\begin{enumerate}
\item[(a)] A basis of $\i_{n,l}$ is given by $\{e_{i,j}t^r \mid 1 \le i,j \le n, \ r \ge s_{i,j} + p_{\min(i,j)}\}$.
\item[(b)] The linear map $\tilde \theta : \cns \to \g^e$ defined by
\begin{equation*}
\tilde \theta(e_{i,j}t^r) = \left\{\begin{array}{cc} 	c_{i,j}^{(r)} & s_{i,j} \le r < s_{i,j} + p_{\min(i,j)} \\
														0 & \text{ otherwise.} \end{array} \right.
\end{equation*}
is a surjective homomorphism of restricted Lie algebras with $\ker \tilde \theta = \i_{n,l}$.
In particular, $\tilde\theta$ induces an isomorphism $\theta : \cnls \isoto \g^e$ of restricted Lie algebras, and a basis of $\cnls$
is given by
\begin{equation}
\label{e:cslbasis}
\{e_{i,j}t^r + \i_{n,l} \mid 1 \le i,j \le n, \, \, \, s_{i,j} \le r < s_{i,j} + p_{\min(i,j)}\}.
\end{equation}
\end{enumerate}
\end{Lemma}

\begin{proof}
Let $\j_{n,l}$ denote the subspace of $\c_n(\sigma)$
with basis $\{e_{i,j}t^r \mid 1 \le i,j \le n, r \ge s_{i,j} + p_{\min(i,j)}\}$.
A straightforward calculation with the
commutator relations in \eqref{e:currentcommutator} shows that $\j_{n,l}$ is in fact an ideal of $\cns$,
and thus we have $\i_{n,l} \sub \j_{n,l}$.

Since $e_{1,1} t^{p_1} \in \i_{n,l}$ and $e_{1,j}t^{s_{1,j}+r} \in \cns$ for $r \ge 0$ we have
$e_{1,j} t^{p_1+s_{1,j}+r} = [e_{1,1}t^{p_1}, e_{1,j} t^{s_{1,j} + r}]$. Similarly
$e_{i,1}t^{p_1+s_{i,1}+r} \in \i_{n,l}$ for $r \ge 0$.
Next we observe $[e_{1,2} t^{p_1+s_{1,2}},e_{2,1} t^{s_{2,1}+r}] = (e_{1,1} - e_{2,2})t^{p_2+r}$ for
$r \ge 0$, where we use that $p_2 = p_1 + s_{1,2} + s_{2,1}$.  Since $p_1 \le p_2$ we have
$e_{1,1}t^{p_2+r} \in \i_{n,l}$, so we can deduce that $e_{2,2}t^{p_2+r} \in \i_{n,l}$ for $r \ge 0$.

By considering the shifted current
Lie algebra spanned by $\{e_{i,j} t^r \mid 2\le i,j \le n,\ r \ge s_{i,j} \}$, and applying
an inductive argument, we obtain that $\j_{n,l} \sub \i_{n,l}$. Hence, $\i_{n,l} = \j_{n,l}$ which proves (a).

The fact that the linear map $\tilde \theta$ in (b) is a homomorphism of restricted Lie algebras
may be seen by comparing the Lie bracket and $p$-power map for
$\g^e$ given in  \eqref{e:centraliserrelations} and \eqref{e:centraliserpstructure}
with those for $\cns$ given in \eqref{e:currentcommutator} and \eqref{e:currentppower}.

It is evident that $e_{1,1} t^r$ lies in the kernel of  $\tilde \theta$ for $r \ge p_1$, so
we have that $\i_{n,l}$ is contained in $\ker \tilde \theta$.   By (a)
we see that \eqref{e:cslbasis}
gives a spanning set
of $\cnls$.  Moreover, by \eqref{e:centraliserbasis} the elements given in \eqref{e:cslbasis} are sent to a basis $\g^e$ by the induced map $\theta : \cnls \to \g^e$.
From this it follows that $\theta$ is an isomorphism, and that \eqref{e:cslbasis} is a basis of $\cnls$.
\end{proof}

\begin{Remark} \label{R:cZform}
For later use we observe that the shifted current algebra and its truncation
can be defined over the integers. We write $\cns_\Z$ for the
free $\Z$-submodule of $\gl_n(\Z) \otimes_\Z \Z[t]$ spanned by the elements
\eqref{e:csbasis}, equipped with its Lie ring structure. We define $\cnls_\Z$
to be the quotient of $\cns_\Z$ by the ideal generated by $\{e_{1,1}t^r \mid r \ge p_1\}$.
Then we observe that the proof of Lemma~\ref{L:truncatedshiftedbasis} can be
applied verbatim to show that $\cnls_\Z$ is a free $\Z$-module spanned by the elements
\eqref{e:cslbasis}.
\end{Remark}

\section{Shifted Yangians and $W$-algebras}
\label{S:YandW}

In this section we fix $n \in \Z_{\ge 1}$, a shift matrix $\sigma = (s_{i,j})$ of size $n$ and
an integer $l > s_{1,n} + s_{n,1}$, which, as usual, we call the level.
We define the pyramid $\pi$ from $(\sigma, l)$ as explained
in \S\ref{ss:comb}.  The partition
$\bp = (p_1,\dots,p_n)$ is defined by \eqref{e:sigmalgivesp}, and we let $N = \sum_{i=1}^n p_i$.
We define $e \in \g = \gl_N(\kk)$ as in \eqref{e:enilp}, and let $G = \GL_N(\kk)$.

\subsection{Shifted Yangians} \label{ss:Yangian}

The {\em shifted Yangian} (over $\kk$) is
the $\kk$-algebra $\Yns$ with generators
\begin{equation} \label{e:Ygens}
\begin{array}{c}
\{D_i^{(r)} \mid 1\le i \le n, \ r>0\} \cup \{E_i^{(r)} \mid 1\le i < n,\ r> s_{i,i+1} \} \\ \cup \, \{F_i^{(r)} \mid 1\le i < n, \ r> s_{i+1,i}\}
\end{array}
\end{equation}
and relations given in \cite[Theorem~4.15]{BT}.  The definition of the shifted Yangian
was first given in \cite{BKshift} over a field of characteristic zero and then
considered in positive characteristic in \cite{BT}.

In order to state the PBW theorem we define the PBW generators
of $Y_n(\sigma)$ as follows. For $i =1,\dots,n-1$ we set
\begin{align*}
E_{i,i+1}^{(r)} &:= E_{i}^{(r)}, \\
F_{i,i+1}^{(r)} &:= F_i^{(r)}
\end{align*}
and define inductively
\begin{align} \label{e:Eij}
\begin{array}{rl}
E_{i,j}^{(r)} := [E_{i, j-1}^{(r - s_{j-1,j})}, E_{j-1}^{(s_{j-1, j} + 1)}] & \text{ for } 1 \le i < j \le n \text{ and } r > s_{i,j}, \\
F_{i,j}^{(r)} := [F_{j-1}^{(s_{j,j-1} + 1)}, F_{i,j-1}^{(r-s_{j, j-1})}] & \text{ for } 1 \le i < j \le n \text{ and } r > s_{j,i}
\end{array}
\end{align}

The {\em loop filtration on $\Yns$} is defined by
placing the elements $E_{i,j}^{(r+1)}, D_i^{(r+1)}, F_{i,j}^{(r+1)}$ in filtered degree $r$ for all $r \ge 0$.
We write
$\cF_r \Yns$ for the filtered piece of degree $r$, so that
$\Yns = \bigcup_{r \ge 0} \cF_r \Yns$, and we write $\gr \Yns$ for the associated
graded algebra.
Then \cite[Lemma~4.13]{BT} says that there is an isomorphism
\begin{equation}
\label{e:isogrTSY}
\tilde \psi : U(\cns) \isoto \gr \Yns
\end{equation}
defined by
\begin{align*}
e_{i,i} t^r &\longmapsto \gr_{r} D_i^{(r+1)}\\
e_{i,j} t^r &\longmapsto \gr_{r} E_{i,j}^{(r+1)}\\
e_{j,i} t^r &\longmapsto \gr_{r} F_{i,j}^{(r+1)}
\end{align*}
for $i < j$.
It follows immediately that the monomials in the elements
\begin{equation} \label{e:pbwgens}
\begin{array}{c}
\{D_i^{(r)} \mid 1 \le i \le n,\ r > 0 \} \cup \{E_{i,j}^{(r)} \mid 1 \le i < j \le n,\ r > s_{i,j} \} \\
\cup  \{F_{i,j}^{(r)} \mid 1 \le i < j \le n,\ r > s_{j,i}\}
\end{array}
\end{equation}
taken in any fixed order give a basis of $Y_n(\sigma)$, and this
gives the PBW theorem for $\Yns$.

\subsection{Finite $W$-algebras}
\label{ss:finiteW}
We move on to introduce of the $W$-algebra $U(\g,e)$,
and begin with the definition stated in \cite{GTmod}.  This is the positive characteristic
analogue of the definition first given by Premet in \cite[Section~4]{PrST}.

Consider the cocharacter $\mu : \kk^\times \to G$ defined by
$\mu(t) = \diag(t^{\col(1)},\dots,t^{\col(n)})$; here we use the
notation $\diag(t_1,\dots,t_N)$ to mean the diagonal $N \times N$ matrix
with $i$th diagonal entry equal to $t_i$.
Using $\mu$ we define the $\Z$-grading
\begin{equation}
\label{e:goodgrading}
\g = \bigoplus_{r \in \Z} \g(r) \quad \text{where} \quad \g(r) := \{x \in \g \mid \mu(t) x = t^r x \text{ for all } t\in \kk^\times\} .
\end{equation}
Since the adjoint action of $\mu(t)$ on a matrix unit is
given by $\mu(t) \cdot e_{i,j} = t^{\col(j) - \col(i)} e_{i,j}$,
we have $\g(r) = \sspan\{e_{i,j} \mid \col(j) - \col(i) = r\}$.

We define the subalgebras
\begin{equation}
\label{e:phandm}
\p :=  \bigoplus_{r \ge 0} \g(r), \quad \h = \g(0), \quad \text{and} \quad
\m := \bigoplus_{r < 0} \g(r)
\end{equation}
of $\g$.
Then $\p$ is a parabolic subalgebra of $\g$ with Levi factor $\h$ and $\m$ is the nilradical of the opposite parabolic to $\p$.
Let $M$ be the closed subgroup of $G$ generated by the root subgroups $u_{i,j}(\kk)$
with $\col(j) < \col(i)$, where $u_{i,j} : \kk \to G$ is defined by $u_{i,j}(t) = 1 + te_{i,j}$.
Then we have $\m = \Lie M$.

We define $\chi \in \g^*$ to be the element dual to $e$ via the trace form on $\g$.
Since $e \in \g(1)$, we have that $\chi$ vanishes on $\g(r)$ for $r \ne -1$.
Therefore, $\chi$ restricts to a character of $\m$.
We define $\m_\chi := \{ x - \chi(x) \mid x\in \m\} \sub U(\g)$, which is a
Lie subalgebra of $U(\g)$.  By the PBW theorem there is a direct sum decomposition
\begin{equation*}
U(\g) = U(\g) \m_\chi \oplus U(\p)
\end{equation*}
and thus a projection
\begin{equation} \label{e:pr}
\pr : U(\g) \to U(\p)
\end{equation}
onto the second factor.
The {\em twisted adjoint action}
of $M$ on $U(\p)$  is defined by
\begin{equation}
\label{e:twistM}
\tw(g) \cdot u := \pr (g \cdot u),
\end{equation}
for $g \in M$ and $u \in U(\p)$; a twisted adjoint action on $S(\p)$ can be defined analogously.
Then the $W$-algebra associated to $e$ is defined to be the invariant subalgebra
\begin{equation*}
U(\g,e) := U(\p)^{\tw(M)} = \{ u \in U(\p) \mid \tw(g) \cdot u = u \text{ for all } g \in M\}.
\end{equation*}

We move on to recall a set of generators for $U(\g,e)$.
These are elements
\begin{equation} \label{e:Wgens}
\begin{array}{c}
\{D_i^{(r)} \mid 1\le i \le n,\ r>0\} \, \cup \, \{E_i^{(r)} \mid 1 \le i < n,\ r> s_{i,i+1} \} \\  \cup  \, \{F_i^{(r)} \mid 1\le i < n,\ r> s_{i+1,i}\}
\end{array}
\end{equation}
in $U(\p)$ defined using the remarkable formulas,
given in \cite[Section 9]{BKshift}; see also \cite[Section 4]{GTmin}.
As is shown in the \cite[Theorem~4.3]{GTmin} the elements in \eqref{e:Wgens} are twisted
$M$-invariants, thus elements of $U(\g,e)$, and moreover they generate $U(\g,e)$.
We note there is an abuse of notation as these generators of $U(\g,e)$ have the same names as the generators
for $\Yns$ given in \eqref{e:Ygens}; this overloading of notation will be justified in the next subsection.

Below we state the formula for $D_i^{(r)}$ in \eqref{e:Dir}, and require some notation for this.
Let $\t$ be the Lie algebra of $T$, and write $\{\epsilon_1,\dots,\epsilon_N\}$
for the standard basis of $\t^*$.
We define the weight $\eta \in \t^*$ by
\begin{equation} \label{e:eta}
\eta := \sum_{i=1}^N (n - q_{\col(i)} - \dots - \cdots - q_l) \epsilon_i,
\end{equation}
where we recall that $q_i$ is the height of the $i$th column in the pyramid $\pi$,
and we note that $\eta$ extends to a character of $\p$.
For $e_{i,j} \in \p$ define
\begin{equation*}
\te_{i,j} :=
e_{i,j} + \eta(e_{i,j}).
\end{equation*}
Then by definition
\begin{equation} \label{e:Dir}
D_i^{(r)} := \sum_{s=1}^r (-1)^{r-s} \sum_{\substack{i_1,\dots,i_s \\ j_1,\dots,j_s}} (-1)^{|\{t=1,\dots,s-1 \mid \row(j_t) \le i-1\}|}
\te_{i_1, j_1} \cdots \te_{i_s, j_s} \in U(\p)
\end{equation}
where the sum is taken over all $1 \le i_1,\dots,i_s,j_1,\dots,j_s \le N$ such that
\begin{itemize}
\item[(a)] $\col(j_1)-\col(i_1)+\dots+\col(j_s)-\col(i_s) + s = r$;
\item[(b)] $\col(i_t) \le \col(j_t)$ for each $t = 1,\dots,s$;
\item[(c)] if $\row(j_t) \ge i$, then $\col(j_t) < \col(i_{t+1})$ for each $t = 1,\dots,s-1$;
\item[(d)] if $\row(j_t) < i$ then $\col(j_t) \ge \col(i_{t+1})$ for each $t = 1,\dots,s-1$;
\item[(e)] $\row(i_1) = i$, $\row(j_s) = i$;
\item[(f)] $\row(j_t) = \row(i_{t+1})$ for each $t = 1,\dots,s-1$.
\end{itemize}
The expressions for the elements $E_i^{(r)} \in U(\p)$ and $F_i^{(r)} \in U(\p)$ are given
by similar formulas; see \cite[Section 9]{BKshift} or \cite[Section 4]{GTmin}.  Then we can define $E_{i,j}^{(r)} \in U(\p)$ and
$F_{i,j}^{(r)} \in U(\p)$ using \eqref{e:Eij}.  As a consequence of the PBW theorem for $U(\g,e)$,
the monomials in
\begin{align}
\label{e:finiteWpbw}
\begin{array}{c}
\{D_i^{(r)} \mid 1 \le i \le n,\ 1 \le r \le p_i\} \, \cup \, \{E_{i,j}^{(r)} \mid
1 \le i< j \le n,\ s_{i,j} \le r \le p_i+ s_{i,j}\}  \\ \cup \,
\{F_{i,j}^{(r)} \mid
1 \le i< j \le n,\ s_{j,i} \le r \le p_i+ s_{j,i}\}
\end{array}
\end{align}
taken in any fixed order form a basis of $U(\g,e)$, see \cite[Lemma~4.2]{GTmin}.

There are two filtrations of the $W$-algebra that we recall here.  First we consider the {\em loop filtration},
which is defined by taking the grading of $U(\p)$ given by the action of the cocharacter
$\mu$, and then the induced filtration $\bigcup_{r=0}^\infty \cF_r U(\g,e)$ of $U(\g,e)$.  We write
$\gr U(\g,e) \sub U(\p)$ for the associated graded algebra.

As a consequence of \cite[Lemma~4.2]{GTmin} we have $D_i^{(r+1)}, E_{i,j}^{(r+1)}, F_{i,j}^{(r+1)}
\in \cF_r U(\g,e)$ and
\begin{align} \label{e:Wgensloop}
\gr_r D_i^{(r+1)} &= (-1)^{r} (c_{i,i}^{(r)} + \eta(c_{i,i}^{(r)})), \\
\gr_r E_{i,j}^{(r+1)} &= (-1)^{r} c_{i,j}^{(r)}, \\
\gr_r F_{i,j}^{(r+1)} &= (-1)^{r} c_{j,i}^{(r)}.
\end{align}
It follows that the shift automorphism on $S_{-\eta} : U(\p) \to U(\p)$ defined by $x \to x - \eta(x)$
for $x \in \p$ restricts to an isomorphism
\begin{equation} \label{e:S-eta}
S_{-\eta}: \gr U(\g,e) \isoto U(\g^e).
\end{equation}

Next we consider the {\em Kazhdan filtration} of $U(\g,e)$.  This is first
defined on $U(\g)$ by placing $x \in \g(r)$ in Kazhdan degree
$r+1$, and as explained in \cite[Section 7]{GTmod},
the associated graded algebra can be identified with $S(\p)$.
We write $\bigcup_{r=0}^\infty \cF'_r U(\g,e)$ of $U(\g,e)$ for the
induced filtration on $U(\g,e)$, and $\gr' U(\g,e)$ for the associated graded
algebra.  Using \cite[Lemma~7.1]{GTmin} we identify $\gr' U(\g,e) = S(\p)^{\tw M}$,
where the twisted adjoint action of $M$ on $S(\p)$ is defined in analogy
with \eqref{e:twistM}.  Further \cite[Lemma~7.1]{GTmod} along with \cite[Lemma~4.2]{GTmin} imply
that the PBW generators $D_i^{(r)}$, $E_{i,j}^{(r)}$ and
$F_{i,j}^{(r)}$ given in \eqref{e:finiteWpbw}
lie in Kazhdan degree $r$ and that $\gr'_r D_i^{(r)}$, $\gr'_r E_{i,j}^{(r)}$ and
$\gr'_r F_{i,j}^{(r)}$ are algebraically independent
generators of $\gr' U(\g,e)$.

\subsection{The truncated shifted Yangian}
\label{ss:truncate}
Our next step is to recall an algebra isomorphism
$\phi$ from a truncation of
the shifted Yangian to the finite $W$-algebra.  This is done in Theorem~\ref{T:isoandPBW}, which also includes the PBW theorem for the truncation.
Although this was proved in \cite[Theorem~4.3]{GTmin}, drawing heavily on the results of \cite{BKshift}, we repeat a few of the details
here to demonstrate that the proof
can be simplified slightly by using the shifted current algebra.

The algebra homomorphism
\begin{equation}
\label{e:phi}
\tilde \phi : \Yns \to U(\g,e)
\end{equation}
is defined by sending the generators $E_i^{(r)}, D_i^{(r)}, F_i^{(r)}$
of $\Yns$ to the generators of $U(\g,e)$ with the same names.
Then we have that $\tilde \phi$ is surjective.  The fact that
this is a homomorphism justifies the abuse of notation in naming the generators of $\Yns$ and $U(\g,e)$.

The {\em truncated shifted Yangian} $\Ynls$
is defined to be the quotient of $\Yns$ by the ideal $I_{n,l}$ generated by the elements $D_1^{(r)}$ with $r > p_1$. It follows directly from formula \cite[(4.2)]{GTmin} that the element $D_1^{(r)}$ of $U(\g,e)$ is equal to zero for $r > p_1$, and so $\tilde \phi$ factors through the quotient
to give a surjection
\begin{equation} \label{e:tildephi}
\phi : \Ynls \to U(\g,e).
\end{equation}
For each element $E_{i,j}^{(r)}, D_i^{(r)}, F_{i,j}^{(r)} \in \Yns$ we write
$\dE_{i,j}^{(r)}, \dD_i^{(r)}, \dF_{i,j}^{(r)}$ for its image in $\Ynls$.

The loop filtration on $\Yns$
descends to a loop filtration on $\Ynls$.  We denote the filtered pieces
by $\cF_r \Ynls$ for $r \ge 0$, and write $\gr \Ynls$ for the associated graded algebra.

We are now ready to show that $\phi$ is an isomorphism and deduce the
PBW theorem for $\Ynls$.

\begin{Theorem}
\label{T:isoandPBW}
$ $
\begin{enumerate}
\item[(a)] $\phi : Y_{n,l}(\sigma) \to U(\g,e)$ is an isomorphism.
\item[(b)] The isomorphism $\tilde \psi : U(\cns) \isoto \gr \Yns$ given in \eqref{e:isogrTSY} induces an isomorphism
\begin{equation} \label{e:psi}
\psi :  U(\cnls) \isoto\gr \Ynls.
\end{equation}
Consequently, the ordered monomials in the elements
\begin{equation}
\label{e:shiftedtruncgens}
\begin{array}{c}
\{\dD_i^{(r)} \mid 1 \le i \le n, \ 0 < r \le p_i \} \cup \{\dE_{i,j}^{(r)} \mid 1 \le i  < j \le n, \ s_{i,j} < r \le s_{i,j} + p_i\} \\
\cup \, \{\dF_{i,j}^{(r)} \mid 1\le i < j \le n, \ s_{j, i} < r \le s_{j,i} + p_i\}
\end{array}
\end{equation}
taken in any fixed order form a basis of $\Ynls$.
\end{enumerate}
\end{Theorem}

\begin{proof}
Using $\tilde \psi$ from \eqref{e:isogrTSY} we identify $U(\cns)$ with $\gr \Yns$.
The associated graded ideal $\gr I_{n,l}$ contains the ideal $\i_{n,l}$ defined
in \S\ref{ss:truncations}, so there is a surjection $U(\cnls) \onto \gr \Ynls$.
It follows from Lemma~\ref{L:truncatedshiftedbasis}(a) and the PBW theorem for $U(\cnls)$
that $\gr \Ynls$ is spanned by the ordered monomials in the elements \eqref{e:shiftedtruncgens}.
Now the PBW theorem for $U(\g,e)$, given in \cite[Theorem~7.2]{GTmod}, along with \cite[Lemma~4.2]{GTmin}
imply that the images under $\phi$ of these spanning elements are linearly independent, and so they form a basis.
This proves (b).

We have seen that $\phi$ sends a basis of $\Ynls$ to a basis of $U(\g,e)$,
so that it is an isomorphism, and we get (a).
\end{proof}

It is helpful for us to give some notation for the PBW basis of $\Ynls$ given by Theorem~\ref{T:isoandPBW}(c).
We fix an order on the sets $\bJ_{F} = \{(i,j,r) \mid 1 \le i < j \le n, s_{j,i} < r \le s_{j,i}+ p_i\}$,
$\bJ_{D} = \{(i,r) \mid 1 \le i \le n, 0 < r \le p_i\}$ and $\bJ_{E} = \{(i,j,r) \mid 1 \le i < j \le n, s_{i,j} < r \le s_{i,j} + p_i\}$.
Let $\bI_F$ be the set of all tuples $\bu = (u_{i,j}^{(r)} \mid (i,j,r) \in \bJ_F)$ of non-negative integers,
$\bI_D$ be the set of all tuples $\bt = (t_i^{(r)} \mid (i,r) \in \bJ_D)$ of non-negative integers, and
$\bI_E$ be the set of all tuples $\bv = (v_{i,j}^{(r)} \mid (i,j,r) \in \bJ_E)$
of non-negative integers.
For $(\bu,\bt,\bv) \in \bI_F \times \bI_D \times \bI_E$, we define
$$
\dF^\bu \dD^\bt \dE^\bv = \prod_{(i,j,r) \in \bJ_{F}} (\dF_{i,j}^{(r)})^{u_{i,j}^{(r)}} \prod_{(i,r) \in \bJ_{D}} (\dD_i^{(r)})^{t_i^{(r)}}
\prod_{(i,j,r) \in \bJ_{E}} (\dE_{i,j}^{(r)})^{v_{i,j}^{(r)}},
$$
where the products respect the orders which we have fixed on $\bJ_{F}$, $\bJ_{D}$ and $\bJ_{E}$.
So that
\begin{equation} \label{e:YnlsPBWbasis}
\{\dF^\bu \dD^\bt \dE^\bv \mid (\bu,\bt,\bv) \in \bI_F \times \bI_D \times \bI_E\}
\end{equation}
is a basis of $\Ynls$.

We can see that the isomorphism $\phi$ in \eqref{e:phi} is filtered for the loop filtration, as
$D_i^{(r+1)}$, $E_{i,j}^{(r+1)}$ and $F_{i,j}^{(r+1)}$ have the same degree, namely $r$, when considered
as elements of $\Yns$ or as elements of $U(\g,e)$.  Thus we obtain
an isomorphism $\gr \phi : \gr \Ynls \isoto \gr U(\g,e)$.
We also have isomorphisms $\psi : U(\cnls) \isoto \gr \Ynls$ from \eqref{e:psi}
and  $U(\theta) : U(\cnls) \isoto U(\g^e)$ given by
Lemma~\ref{L:truncatedshiftedbasis}, and we have the isomorphism $S_{-\eta} : \gr U(\g,e) = U(\g^e)$.
We note however that as isomorphisms $\gr \Ynls \isoto U(\g^e)$, we have
$U(\theta) \circ \psi^{-1} \ne S_{-\eta} \circ \gr \phi$.
To explain this we note the adjoint action of $\mu(-1)$ gives an
automorphism $U(\g^e) \to U(\g^e)$, which is determined by
$c_{i,j}^{(r)} \mapsto (-1)^r c_{i,j}^{(r)}$; here
we recall that $\mu$ is the cocharacter defining the good grading on $\g$.
Then we have
\begin{equation} \label{e:yuck}
\Ad(\mu(-1)) \circ U(\theta) \circ \psi^{-1} = S_{-\eta} \circ \gr \phi.
\end{equation}
For Theorem~\ref{T:centreY}, we need to fix an isomorphism between
$\gr \Ynls \isoto U(\g^e)$.  For consistency with \cite{BB}, we use
\begin{equation}
 \label{e:grYnlsge}
S_{-\eta} \circ \gr \phi : \gr \Ynls \isoto U(\g^e)
\end{equation}
which is determined by its effect on the generators as follows
\begin{align*}
\gr_{r} \dD_i^{(r+1)} &\mapsto (-1)^r c_{i,i}^{(r)}\\
 \gr_{r} \dE_{i,j}^{(r+1)} &\mapsto (-1)^r c_{i,j}^{(r)}\\
\gr_{r} \dF_{i,j}^{(r+1)} &\mapsto (-1)^r c_{j,i}^{(r)}
\end{align*}

\subsection{The integral forms of $\Yns$ and $\Ynls$} \label{ss:Zforms}
We introduce and study the integral (truncated) shifted Yangian.
There are two natural approaches: we can consider the subring of the complex (truncated) shifted Yangian
generated by the elements listed in \eqref{e:Ygens};
or we can consider the ring determined by these generators and the relations in \cite[Theorem 4.15]{BT}
(along with the relations $D_1^{(r)} = 0$ for $r > p_1$).
Lemmas~\ref{L:integralshiftedPBW} and \ref{L:integraltruncatedPBW}
say that these two approaches lead to isomorphic rings.
As explained by Corollary~\ref{C:redmodp} this allows us to apply reduction modulo $p$ to certain formulas
in the complex truncated shifted Yangian, which will be useful later on.

Let $A$ be a commutative ring.  We define
the {\em shifted $A$-Yangian} $Y_n^A(\sigma)$ to be the $A$-algebra
with generators given in \eqref{e:Ygens} subject to the relations in \cite[Theorem~4.15]{BT}.
Here we are only concerned with the cases where $A = \Z$, $\C$ or $\kk$.  We note that
$Y_n^\kk(\sigma) = \Yns$, and that $Y_n^\C(\sigma)$ is the usual complex shifted Yangian,
as considered in \cite{BKshift} and \cite{BKrep}.
We mildly abuse notation by viewing the elements in \eqref{e:Ygens} simultaneously
as elements of $Y_n^\Z(\sigma)$,  $Y_n^\C(\sigma)$ and $\Yns$.

There is a ring homomorphism $Y_n^\Z(\sigma) \to Y_n^\C(\sigma)$
sending a generator of $Y_n^\Z(\sigma)$ to the element of $Y_n^\C(\sigma)$ with the same name.  This induces
a ring homomorphism $Y_n^\Z(\sigma) \otimes_\Z \C \to Y_n^\C(\sigma)$.  Similarly,
there is a natural map $Y_n^\Z(\sigma) \otimes_\Z \kk \to Y_n(\sigma)$.

\begin{Lemma}
\label{L:integralshiftedPBW}
$ $
\begin{enumerate}
\item[(a)] $Y_n^\Z(\sigma)$ is a free $\Z$-module with basis given by ordered monomials in the elements
given in \eqref{e:pbwgens}.
\item[(b)] The homomorphism $Y_n^\Z(\sigma) \to Y_n^\C(\sigma)$ is injective and the induced homomorphism
$Y_n^\Z(\sigma) \otimes_\Z \C \isoto Y_n^\C(\sigma)$ is an isomorphism.
\item[(c)] The homomorphism $Y_n^\Z(\sigma) \otimes_\Z \kk \isoto Y_n(\sigma)$ is an isomorphism.
\end{enumerate}
\end{Lemma}

\begin{proof}
As introduced in Remark~\ref{R:cZform} we write the $\cns_\Z$ for
$\Z$-form of $\cns$.
The argument in the penultimate paragraph of the proof of
\cite[Theorem~4.3]{BT} can be applied verbatim to show that there is a surjection
$U(\cns_\Z) \onto Y_n^\Z(\sigma)$: to apply this argument it is
necessary to define a loop filtration on $Y_n^\Z(\sigma)$, which can be done
by placing $E_{i,j}^{(r+1)}, D_i^{(r+1)}, F_{i,j}^{(r+1)}$ in degree $r$.
We can now deduce that the PBW monomials in the elements in \eqref{e:pbwgens} form a
spanning set of $Y_n^\Z(\sigma)$ over $\Z$.
By \cite[Theorem~2.1]{BKshift} these monomials are sent to $\C$-linearly
independent elements of $Y_n^\C(\sigma)$ under the map
$Y_n^\Z(\sigma) \to Y_n^\C(\sigma)$.  Therefore, they are certainly $\Z$-linearly independent in $Y_n^\Z(\sigma)$.
This proves (a).  Also we have shown that $Y_n^\Z(\sigma) \to Y_n^\C(\sigma)$ sends a $\Z$-basis to a $\C$-basis, which implies (b).

Thanks to \cite[Theorem~4.14]{BT}, ordered monomials in the elements in \eqref{e:pbwgens} form
a $\kk$-basis of $\Yns$.  Thus the map $Y_n^\Z(\sigma) \otimes_\Z \kk \to \Yns$ sends to a $\kk$-basis of
$Y_n^\Z(\sigma) \otimes_\Z \kk$ to a $\kk$-basis of $\Yns$, and we obtain (c).
\end{proof}

We also want an analogue of Lemma~\ref{L:integralshiftedPBW}
in the context of truncated shifted Yangians.  To do this we first define the {\em truncated shifted $A$-Yangian}
$Y_{n,l}^A(\sigma)$ to be the quotient of $Y_n^A(\sigma)$ by the ideal generated by $\{D_1^{(r)} \mid r > p_1\}$.
Similarly to the non-truncated case we have maps $Y_{n,l}^\Z(\sigma) \to Y_{n,l}^\C(\sigma)$,
$Y_{n,l}^\Z(\sigma) \otimes_\Z \C \to Y_{n,l}^\C(\sigma)$ and $Y_{n,l}^\Z(\sigma) \otimes_\Z \kk \to Y_{n,l}(\sigma)$.

\begin{Lemma}  \label{L:integraltruncatedPBW}
$ $
\begin{enumerate}
\item[(a)] $Y_{n,l}^\Z(\sigma)$ is a free $\Z$-module with basis given in \eqref{e:YnlsPBWbasis}.
\item[(b)] The homomorphism $Y_{n,l}^\Z(\sigma) \to Y_{n,l}^\C(\sigma)$ is injective and the induced homomorphism
$Y_{n,l}^\Z(\sigma) \otimes_\Z \C \isoto Y_{n,l}^\C(\sigma)$ is an isomorphism.
\item[(c)] The homomorphism $Y_{n,l}^\Z(\sigma) \otimes_\Z \kk \isoto Y_{n,l}(\sigma)$ is an isomorphism.
\end{enumerate}
\end{Lemma}

\begin{proof}
Recall that  $\cnls_\Z$ is defined in Remark~\ref{R:cZform}.
The argument at the start of the
proof of Theorem~\ref{T:isoandPBW} can be applied to show that
$U(\cnls_\Z)$ surjects onto $\gr Y_{n,l}^\Z(\sigma)$.

Now we can complete the proof of the
current lemma using the same steps as in the proof of Lemma \ref{L:integralshiftedPBW},
employing the PBW theorems for $Y_{n,l}^\C(\sigma)$ and $Y_{n,l}(\sigma)$,
which are given in \cite[Corollary 6.3]{BKshift} and Theorem~\ref{T:isoandPBW}.
\end{proof}

Thanks to the previous lemma
we can employ reduction modulo $p$ to deduce formulas in $Y_{n,l}(\sigma)$ from certain
types of formulas in $Y_{n,l}^\C(\sigma)$, as explained by the following corollary.

\begin{Corollary} \label{C:redmodp}
Let $h$ be a polynomial with coefficients in $\Z$ in the
non-commuting indeterminates
$\{f_i^{(r)} \mid 1 \le i < n, r > s_{i+1,i}\} \cup \{d_j^{(r)} \mid 1 \le j \le n, r >0\} \cup
\{e_i^{(r)} \mid 1 \le i < n, r > s_{i,i+1}\}$, and let $A$ be a ring.
Write $H^A$ for the element of $Y_{n,l}^A(\sigma)$ obtained by specialising $h$ via $d_i^{(r)} \mapsto \dD_i^{(r)}$, \, $e_i^{(r)} \mapsto \dE_i^{(r)}$ and $f_i^{(r)} \mapsto \dF_i^{(r)}$.
\begin{enumerate}
\item[(a)] Suppose that $H^\C = 0$.  Then $H^\kk = 0$.
\item[(b)] Suppose that $H^\C \in \cF_r Y_{n,l}^\C(\sigma)$.
Then  $H^\kk \in \cF_r Y_{n,l}(\sigma)$.
\end{enumerate}
\end{Corollary}

\begin{proof}
By Lemma~\ref{L:integraltruncatedPBW}(b) we can view $Y_{n,l}^\Z(\sigma) \sub Y_{n,l}^\C(\sigma)$.
Then we have that $H^\C = H^\Z \in Y_{n,l}^\Z(\sigma)$, so that
$H^\Z = 0$.  Further, under the identification $\Ynls \iso Y_{n,l}^\Z(\sigma) \otimes_\Z \kk$,
given by Lemma~\ref{L:integraltruncatedPBW}(c),
we have $H^\kk = H^\Z \otimes 1$.  Hence, $H^\kk = 0$, and this proves (a)

Using Lemma~\ref{L:integraltruncatedPBW}(a) may write $H^\Z \in Y_{n,l}^\Z(\sigma)$ as a $\Z$-linear combination
of the PBW basis given in \eqref{e:YnlsPBWbasis}
given by monomials in the elements from \eqref{e:shiftedtruncgens} in some fixed order.
This also gives the expression for $H^\C$ in terms of this PBW basis
in $Y_{n,l}^\C(\sigma)$ and for $H^\kk$ in terms of this PBW basis in $\Ynls$.
Furthermore, the filtered degree for the loop filtration can be read off directly
from these expressions, which implies (b).
\end{proof}

The observations of the previous lemma will be convenient
for us at several places later in this paper.  However we should mention that
we expect that the formulas which we verify using this approach can also be established
over $\kk$ by repeating the known methods over $\C$. Thus the reduction modulo $p$
procedure may be viewed as a convenient alternative to reciting certain
technical arguments from characteristic zero.

We end this subsection by explaining that some parts of the
the theory of $U(\g,e)$ from \S\ref{ss:finiteW} can be carried out over $\Z$.
Let $\p_\Z$ be the parabolic subalgebra $\g_\Z = \gl_n(\Z)$ such that $\p = \p_\Z \otimes_\Z \kk$.
The element $\chi \in \g^*$, can be viewed as a function from $\g_\Z \to \Z$, and then
we can define a projection
\begin{equation} \label{e:prZ}
\pr_\Z : U(\g_\Z) \to U(\p_\Z)
\end{equation}
in analogy with $\pr$
as defined in \eqref{e:pr}.

The isomorphism $\phi : \Ynls \isoto U(\g,e)$ from \eqref{e:phi} can be thought
of as an embedding $\phi : \Ynls \hookrightarrow U(\p)$.
By considering the formulas for of the twisted $M$-invariants $D_i^{(r)}, E_i^{(r)}, F_i^{(r)} \in U(\p)$ given in
\cite[Section 9]{BKshift}, see also \cite[Section 4]{GTmin},
we see that they can be viewed as an elements of $U(\p_\Z)$.  Therefore,
we can consider the ring homomorphism defined in the obvious manner
\begin{equation} \label{e:phiZ}
\phi_\Z : Y_{n,l}^\Z(\sigma) \to U(\p_\Z).
\end{equation}

We also note here that the procedure of the reduction modulo $p$ given by
Corollary~\ref{C:redmodp} has an obvious analogue with
$U(\p_\C)$, $U(\p)$ and $U(\p_\Z)$ in place of  $Y_{n,l}^\C(\sigma)$, $\Ynls$
and $Y_{n,l}^\Z(\sigma)$; these observations will be vital in the proof of
Lemma~\ref{L:HClemmaUg2}.

\subsection{The $T_{i,j}^{(r)}$ generators for $\Ynls$}
We introduce some alternative PBW generators, which will be important later.
They were described in \cite[Section~2.2]{BKrep} over $\C$.  We recap the details for the readers convenience.

Let $u$ be an indeterminate, and consider the power series ring $\Yns[[u^{-1}]]$.
We adopt the convention $D_i^{(0)} = 1$ for all $i$
and define the power series
\begin{align} \label{e:powerseries}
D_i(u), E_{i,j}(u), F_{i,j}(u) \in \Yns[[u^{-1}]]
\end{align}
by setting $D_i(u) = \sum_{r\ge 0} D_i^{(r)} u^{-r}$, $E_{i,j}(u) = \sum_{r > s_{i,j} } E_{i,j}^{(r)} u^{-r}$ and $F_{i,j}(u) = \sum_{r > s_{j,i}} F_{i,j}^{(r)} u^{-r}$. By convention we also set $E_{i,i}(u) = F_{i,i}(u) = 1$.

Next we define the following $n \times n$ matrices with coefficients in $\Yns[[u^{-1}]]$:
\begin{itemize}
\item $D(u)$ is the diagonal matrix with $D(u)_{i,i} =D_i(u)$,
\item $E(u)$ is the upper unitriangular matrix with
$E(u)_{i,j}:=E_{i,j}(u)$ for $i \le j$,
\item $F(u)$ is the lower unitriangular matrix with
$F(u)_{i,j}:=F_{j,i}(u)$ for $i \ge j$.
\end{itemize}
Now define the matrix $T(u) = F(u)D(u)E(u)$, whose $(i,j)$-entry can be written as a power series
\begin{equation} \label{e:Tij}
T_{i,j}(u) = \sum_{r \ge 0} T_{i,j}^{(r)} u^{-r} := \sum_{k=1}^{\min(i,j)} F_{k,i}(u) D_k(u) E_{k, j}(u)
\end{equation}
for some elements $T_{i,j}^{(r)} \in \Yns$. The image of $T_{i,j}^{(r)}$ in $\Ynls$ will be denoted $\dT_{i,j}^{(r)}$.

By direct calculation we easily see that $T_{i,j}^{(0)} = \delta_{i,j}$ and $T_{i,j}^{(r)} = 0$ for $0 < r \le s_{i,j}$.
Also we can see that $T_{i,j}^{(r+1)} \in \cF_r \Yns$ and then using the isomorphism $\tilde \psi$ from \eqref{e:isogrTSY}
to identify $\gr \Yns \iso U(\cns)$ we have
$$
\gr_r T_{i,j}^{(r+1)} = e_{i,j} t^r.
$$
This allows us to deduce that the $T_{i,j}^{(r)}$ give alternative
PBW generators as stated in the next lemma; the version of these results
in characteristic zero are given in \cite[Lemmas 2.1 and 3.6]{BKrep}.

\begin{Lemma}
\label{L:propertiesofTij}
$ $
\begin{enumerate}
\item[(a)] The ordered monomials in the elements
$\{T_{i,j}^{(r)} \mid 1\le i, j \le n, \ s_{i,j} < r\}$ form a basis for $\Yns$.
\item[(b)] The ordered monomials in the elements
$\{\dT_{i,j}^{(r)} \mid 1\le i, j \le n, \ s_{i,j} < r \le s_{i,j} + p_{\min(i,j)}\}$ form a basis for $\Ynls$.
\end{enumerate}
\end{Lemma}

The next result is obtained as application of reduction modulo $p$, using Corollary
\ref{C:redmodp}.

\begin{Corollary}
\label{C:Tij0}
$\dot T_{i,j}^{(r)} = 0$ in $\Ynls$ for $r > p_{\min(i,j)}+s_{i,j}$
\end{Corollary}

\begin{proof}
The version of this statement in $Y_{n,l}^\C(\sigma)$ is
\cite[Theorem 3.5]{BKrep}.  Now we can apply Corollary
\ref{C:redmodp}.
\end{proof}

\section{Centres and restricted versions} \label{S:cent}

In this section we study the centres of $Y_{n,l}(\sigma)$ and $U(\g,e)$. Both algebras
admit a natural definition of a Harish-Chandra centre and a $p$-centre arising in different
ways, and we show that in either case the centre is generated by these subalgebras.
We continue to use the notation from Section~\ref{S:YandW}.

\subsection{The centre of $\Yns$} \label{ss:centreshiftedYangian}
We proceed to recall the description of the centre of $Y_n(\sigma)$ given in \cite{BT}.
The power series $D_i(u)$ are defined in \eqref{e:powerseries}.
From these we define
\begin{align*}
C(u) = \sum_{r \ge 0} C^{(r)} u^{-r} &:= D_1(u) D_2(u-1) D_3(u-2) \cdots D_n(u-n+1).
\end{align*}
By \cite[Theorem 5.11(1)]{BT}, the elements in $\{C^{(r)} \mid r > 0\}$ are algebraically independent and
lie in the centre $Z(\Yns)$ of $\Yns$.
The subalgebra they generate
is called {\em the Harish-Chandra centre of $\Yns$}, and is denoted $Z_\HC(\Yns)$.

For $i=1,\dots,n$, we define
\begin{equation}
\label{e:Bir}
B_i(u) = \sum_{r \ge 0} B_i^{(r)}  := D_i(u) D_i(u-1) D_i(u-2) \cdots D_i(u-p+1).
\end{equation}
By \cite[Theorem 5.11(2)]{BT} the elements in
\begin{equation}
\label{e:pgenerators}
\begin{array}{c}\{B_i^{(rp)} \mid i = 1,\dots,n, r > 0\} \, \cup \\
\{(E_{i,j}^{(r)})^p \mid 1\le i < j \le n, r > s_{i,j}\} \cup \{(F_{i,j}^{(r)})^p \mid 1\le i< j \le n, r >s_{j,i}\}\end{array}
\end{equation}
are algebraically independent, and lie in $Z(\Yns)$.
The subalgebra they generate is called the {\em $p$-centre} of $\Yns$ and is denoted $Z_p(\Yns)$.
We note that by \cite[Theorem 5.8]{BT}, the elements $B_i^{(s)}$ can be written in terms of $B_i^{(rp)}$ for
$0 \le r \le \frac{s}{p}$, so in particular they lie in the $p$-centre of $\Yns$.
Furthermore by \cite[Theorems~5.1, 5.4, 5.8]{BT} we have
\begin{equation}
\label{e:degreeofBEF}
B_i^{((r+1)p)},
(E_{i,j}^{(r+1)})^p, (F_{i,j}^{(r+1)})^p \in \cF_{rp} \Yns,
\end{equation}
and
under the identification of $\gr \Yns \iso U(\cns)$ given by the isomorphism $\tilde \psi$
from \eqref{e:isogrTSY} we have
\begin{align} \label{e:tops}
\gr_{rp} B_i^{((r+1)p)} &= (e_{i,i}t^{r})^p - e_{i,i} t^{pr} \in Z_p(\cns); \nonumber  \\
\gr_{rp} (E_{i,j}^{(r+1)})^p &= (e_{i,j} t^{r})^p \in Z_p(\cns); \\
\gr_{rp} (F_{i,j}^{(r+1)})^p &= (e_{j,i} t^{r})^p \in Z_p(\cns). \nonumber
\end{align}
From this it follows that $Z_p(\Yns)$ is a polynomial algebra
over the generators given in \eqref{e:pgenerators}.

Though we do not require it in this paper we remark that
\cite[Theorem 5.11]{BT} contains more information about the centre of $\Yns$.
In particular, it is generated by $Z_{HC}(\Yns)$ and $Z_p(\Yns)$.
We also mention that \cite[Corollary 5.13]{BT} states that $\Yns$ is a free module
over $Z_p(\Yns)$ with basis given by the  by the ordered monomials in
the generators in \eqref{e:pbwgens} in which no exponent is $p$ or more;
we refer to such monomials as {\em $p$-restricted monomials}.

\subsection{The centre of the truncated shifted Yangian}
\label{ss:centreTSY}

In this subsection we prove Theorem~\ref{T:centreY}, giving a precise description of the centre
of $\Ynls$. As in \cite[Lemma 3.7]{BKrep} we define the Laurent series
\begin{equation}
\label{e:defineZu}
Z(u) = \sum_{r \ge 0} Z_r u^{N-r} := u^{p_1} (u-1)^{p_2} \cdots (u-(n-1))^{p_n} C(u) \in \Yns((u^{-1})).
\end{equation}
Following the convention established in \S\ref{ss:Yangian} we use the
notation $\dB_i^{(r)}, \dC^{(r)}, \dZ_r$
to denote the images of
$B_i^{(r)}, C^{(r)}, Z_r
\in \Yns$ in the quotient $\Ynls$; similarly we use the
power series notation $\dC(u), \dZ(u)$.

\begin{Lemma}
$\dZ(u) = u^N + \sum_{r=0}^N \dZ_r u^{N-r} \in \Ynls[u]$ is a polynomial in $u$ of degree $N$.
\end{Lemma}

\begin{proof}
We may view $\dZ(u)$ as a Laurent series in $u^{-1}$ with coefficients in
the complex truncated shifted Yangian
$Y_{n,l}^\C(\sigma)$, which can be expressed
as an integral linear combination of products of the generators of $\Ynls$.  In this setting
\cite[Lemma 3.7]{BKrep} implies that $\dZ(u)$
is in fact a polynomial in $u$ of degree $N$.
Now viewing $\dZ(u)$ as a Laurent series in $u^{-1}$ with coefficients in $\Ynls$ and using
Corollary \ref{C:redmodp}, we deduce
that $\dZ(u)$ is a polynomial in $u$ of degree $N$.
\end{proof}

Examining the coefficient of $u^{N-r}$ in \eqref{e:defineZu} we see that for
$r > N$ the element $\dC^{(r)} \in \Ynls$ is a linear combination of $\dC^{(r-N)},\dots,\dC^{(r-1)}$. In particular, the image of $Z_{\HC}(\Yns)$ in $\Ynls$,
is generated by $\{\dC^{(r)} \mid r = 1,\dots,N\}$ or equivalently by
$\{\dZ_r \mid r = 1,\dots,N\}$.  We refer to this subalgebra of $Z(\Ynls)$ as
{\em the Harish-Chandra centre of $\Ynls$} and denote it by $Z_{\HC}(\Ynls)$.

Similarly we define {\em the $p$-centre of $\Ynls$} to be the image of the $p$-centre of $\Yns$ in $\Ynls$, and denote it by $Z_p(\Ynls)$.

We are now ready to state and prove our description of the centre of $\Ynls$.

\begin{Theorem}
\label{T:centreY}
$ $
\begin{enumerate}
\item[(a)] The elements $\dZ_1,\dots,\dZ_N$ are algebraically independent generators for $Z_\HC(\Ynls)$.
\item[(b)] The elements of
\begin{equation}
\label{e:pgenerators2}
\begin{array}{c}\{\dot B_i^{(rp)} \mid 1\le i \le n, \ 0 < r \le p_i\} \, \cup \\
\{(\dot E_{i,j}^{(r)})^p \mid 1\le i < j \le n,\ s_{i,j} < r \le s_{i,j} + p_{\min(i,j)}\} \, \cup \\
\{(\dot F_{i,j}^{(r)})^p \mid 1\le i< j \le n,\ 0 < r \le s_{j,i} + p_{\min(i,j)}\}\end{array}
\end{equation}
are algebraically independent generators of $Z_p(\Ynls)$.
\item[(c)] Via the isomorphism $\gr \Ynls \iso U(\g^e)$ given in \eqref{e:grYnlsge}, we have that
$\gr Z_p(\Ynls)$ identifies with $Z_p(\g^e) \sub U(\g^e)$.
\item[(d)] $Z(\Ynls)$ is a free module of rank $p^N$ over $Z_p \Ynls$: a basis is given by
\begin{equation}
\label{e:freebasiscentreTSY}
\{\dZ_1^{k_1} \cdots \dZ_N^{k_N} \mid 0 \le k_i < p\}.
\end{equation}
\end{enumerate}
\end{Theorem}

\begin{proof}
We first prove the theorem under the assumption that $\sigma$ is upper-triangular, and then explain
how to deduce it in general.  So assume for now that $\sigma$ is upper-triangular.

We begin by giving an alternative expression for $C(u) \in \Yns[[u^{-1}]]$.
Recall that column determinants are defined in \eqref{e:cdetdefn}, and the
power series $T_{i,j}(u)$ are defined in \eqref{e:Tij}.
Viewing $C(u)$ as an element of $Y_{n}^\C(\sigma)[[u^{-1}]]$ we have
\begin{align}
\label{e:Cisacdet}
C(u) = \cdet \left( \begin{array}{cccc}  T_{1,1}(u) &  T_{1,2}(u-1) & \dots &  T_{1,n}(u-n+1)\\
T_{2,1}(u) &  T_{2,2}(u-1) & \dots &  T_{2,n}(u-n+1)\\
\vdots & \vdots & \ddots & \vdots \\
 T_{n,1}(u) &  T_{n,2}(u-1) & \dots &  T_{n,n}(u-n+1) \end{array}\right) \in \Yns.
\end{align}
as a consequence of \cite[(2.79)]{BKrep} and \cite[Theorem~2.2]{BB}.  Some further explanation of this
is appropriate, as \cite[(2.79)]{BKrep} shows that \eqref{e:Cisacdet} holds for the (unshifted) Yangian $Y_{n}^\C$.
However, as explained by \cite[Corollary~2.2]{BKshift}, we can view $Y_{n}^\C(\sigma) \sub Y_{n}^\C$, and
then \cite[Theorem~2.2]{BB} implies that \eqref{e:Cisacdet} holds for $Y_{n}^\C(\sigma)$.  A subtle point
here is that the elements $T_{i,j}^{(r)} \in Y_{n}^\C(\sigma)$ depend on $\sigma$, as is explained in \cite[Section 2]{BB},
and this is where we require that $\sigma$ is upper triangular.
Now using Corollary~\ref{C:redmodp} we have that \eqref{e:Cisacdet} holds in $Y_n(\sigma)[[u^{-1}]]$.

For the next step we claim that $\dZ_r \in \cF_{r-d_r} Y_{n,l}(\sigma)$ and that
$\gr_{r-d_r} \dZ_r = (-1)^{r-d_r} z_r \in \gr \Ynls = U(\g^e)$, under the identification
$\gr \Ynls \iso U(\g^e)$ given by \eqref{e:grYnlsge}.
Thanks to \eqref{e:Cisacdet}, the definition of $\dZ(u)$ given in \eqref{e:defineZu} is the same
as that given in \cite[(3.2)]{BB}.
Next we observe that the formula given in \cite[Lemma~3.5]{BB} which expresses $\dZ_r$ in terms of
the elements $\dT_{i,j}^{(r)}$ can be expressed as an integral linear combination of products of the generators of
of $Y_{n,l}^\C(\sigma)$ in \eqref{e:Ygens}.  Applying Corollary~\ref{C:redmodp}
we conclude that the same formula holds for  $\dZ_r \in \Ynls$. Now the argument used to
complete the proof of \cite[Theorem~3.4]{BB} can be repeated verbatim to deduce the claims
made at the beginning of this paragraph.

Now we may combine Lemma~\ref{L:Ucentralisercentre}(a)
with a standard filtration argument to deduce that $\dZ_1,\dots,\dZ_N$ are algebraically independent, proving (a).

Let $i=1,\dots,n$ and $0 \le r < p_i$. According to \eqref{e:degreeofBEF} and \eqref{e:tops}
the element $B_i^{((r+1)p)} \in \Yns$  lies in loop degree $rp$ and  $\gr_{rp} B_i^{((r+1)p)} = (e_{i,i}t^{r})^p - e_{i,i} t^{rp}$,
under the identification $\gr \Yns \iso U(\cns)$ given by $\tilde \psi$ from \eqref{e:isogrTSY}.
More explicitly this means that $\tilde \psi((e_{i,i}t^{r})^p - e_{i,i} t^{rp}) = \gr_{rp} B_i^{((r+1)p)}$.
Using Lemmas~\ref{L:truncatedshiftedbasis} and \ref{L:integralshiftedPBW}, we
deduce that $\dB_i^{((r+1)p)}$ has loop degree $rp$ in $\Ynls$ and
that $\psi((e_{i,i}t^{r})^p - e_{i,i} t^{rp} + \i_{n,l}) = \gr_{rp} \dB_i^{((r+1)p)}$, where
$\psi$ is defined in \eqref{e:psi}.
Now using \eqref{e:yuck}, we see that $\gr_{rp} \dB_i^{((r+1)p)} = (-1)^{rp}(c_{i,i}^{(r)})^p - (-1)^{rp}c_{i,i}^{(rp)} \in Z_p(\g^e)$
under the identification $\gr \Ynls \iso U(\g^e)$ given by \eqref{e:grYnlsge}.

A similar argument shows that the elements $(\dE_{i,j}^{(r+1)})^p, (\dF_{i,j}^{(r+1)})^p \in \Ynls$
lie in loop degree $rp$ and satisfy $\gr_{rp} (\dE_{i,j}^{(r+1)})^p = (-1)^{rp} (c_{i,j}^{(r)})^p$
and $\gr_{rp} (\dF_{i,j}^{(r+1)})^p = (-1)^{rp} (c_{j,i}^{(r)})^p$, under the identification
$\gr \Ynls \iso U(\g^e)$. Since these elements are algebraically independent
generators for $Z_p(\g^e)$, it follows that the elements in \eqref{e:pgenerators2} are algebraically
independent in $Z_p(\Ynls)$.

We next show that $Z_p(\Ynls)$ coincides with the
algebra generated by the elements in \eqref{e:pgenerators2}; we denote this latter algebra by $\widehat Z_p(\Ynls)$.

From the pyramid $\pi$ associated to $(\sigma, l)$ we construct the pyramid $\bar \pi$
by adding another row to the bottom of length $p_n$, as we did in
\S\ref{ss:Znilpcent}. This gives a new shift matrix $\bar\sigma$ with $\bar s_{n,n+1} = \bar s_{n+1, n} = 0$ and
$s_{i,i+1} = \bar s_{i,i+1}$, $s_{i+1, i} = \bar s_{i+1, i}$ for $i=1,\dots,n$.
The defining relations of the truncated shifted Yangian, along with the PBW theorem given in Theorem~\ref{T:isoandPBW}(b)
imply that there is an embedding
$\Ynls \into Y_{n+1,l}(\bar\sigma)$. Since the elements $(\dE_{i,j}^{(r)})^p, (\dF_{i,j}^{(r)})^p, \dB_i^{(rp)} \in \Ynls$ are sent to the elements of $Y_{n+1,l}(\bar\sigma)$ with the same names, it follows that these elements are central in $Y_{n+1, l}(\bar\sigma)$.
We conclude that every element of $Z_p(\Ynls)$ commutes with every element of
$Y_{n+1,l}(\bar\sigma)$. Following the notation of Lemma~\ref{L:gepcent}
we identify $\gr Y_{n+1,l}(\bar\sigma)$ with $U(\bar \g^{\bar e})$ using
the analogue of the isomorphism given in \eqref{e:grYnlsge}.

We will show that the inclusion $\widehat Z_p(\Ynls) \sub Z_p(\Ynls)$ is an
equality by considering the associated graded algebras. Thanks to our previous observations we have
$\gr \widehat Z_p(\Ynls) = Z_p(\g^e)$. Suppose that
$Z_p(\Ynls) \setminus \widehat Z_p(\Ynls) \neq \varnothing$ and choose an element
$u$ of minimal loop degree, say $d$. By the remarks of the previous paragraph we see that
$\gr_d u$ commutes with everything in $U(\bar \g^{\bar e})$ and applying Lemma~\ref{L:gepcent}
we see that $\gr_d u \in Z_p(\g^e)$. As we observed above the generators of $Z_p(\g^e)$
are all of the form $\gr_{(r-1)p} B_i^{(rp)}$, $\gr_{(r-1)p} (E_{i,j}^{(r)})^p$, $\gr_{(r-1)p} (F_{i,j}^{(r)})^p$
where the indexes $i,j,r$ are restricted in accordance with \eqref{e:pgenerators2}.
Consequently there exists $u' \in \widehat Z_p(\Ynls)$ of loop degree $d$ such
that $\gr_d u = \gr_d u'$. Since $u \notin \widehat Z_p(\Ynls)$ we deduce that
$u - u' \in Z_p(\Ynls) \setminus \widehat Z_p(\Ynls)$ is of strictly lower loop degree.
Since the degree of $Z$ was assumed to be minimal, we have reached a contradiction.  This confirms
that $Z_p(\Ynls) = \widehat Z_p(\Ynls)$, and thus completes the proof of (b).

To prove (c), we start by observing that we have shown
\begin{equation} \label{e:grpr}
\begin{array}{c}
\gr_{rp} \dB_i^{((r+1)p)} = (-1)^{rp}(c_{i,i}^{(r)})^p - (-1)^{rp}c_{i,i}^{(rp)}, \quad
\gr_{rp} (\dE_{i,j}^{(r+1)})^p = (-1)^{rp} (c_{i,j}^{(r)})^p \\ \text{and} \quad \gr_{rp} (\dF_{i,j}^{(r+1)})^p = (-1)^{rp} (c_{j,i}^{(r)})^p
\end{array}
\end{equation}
generate both $\gr Z_p(\Ynls)$ and $Z_p(\g^e)$.  Hence, $\gr Z_p(\Ynls) = Z_p(\g^e)$.

We have seen that $\gr_{r - d_r} Z_r = (-1)^r z_r$, and we have
also have \eqref{e:grpr}.
Thus Lemma~\ref{L:Ucentralisercentre} along with a standard filtration argument implies that
$Z(\Ynls)$ is generated by $Z_\HC(\Ynls)$ and $Z_p(\Ynls)$.
Now we can deduce (d)
from Lemma~\ref{L:filtfree} and Lemma~\ref{L:Ucentralisercentre}(b).

We have now completed the proof in case $\sigma$ is upper-triangular and it
remains to explain how to deduce the theorem for arbitrary $\sigma$.
First we note that our proof of (b) and (c) does not actually require the assumption
that $\sigma$ is upper triangular.  So we are left to deal with (a) and (d).

It follows from \cite[4.5, (4)]{BT} that there exists
an upper-triangular shift matrix $\sigma_u$ and an isomorphism
$\iota : Y_n(\sigma) \isoto Y_n(\sigma_u)$.
Each of these algebras has a commutative subalgebra
generated by $\{D_i^{(r)} \mid 1 \le i \le n, r \ge 0\}$, and
the isomorphism $\iota$ fixes this subalgebra pointwise.
Consequently, there is an induced isomorphism
$Y_{n,l}(\sigma) \isoto Y_{n,l}(\sigma_u)$.
This same fact also shows that the coefficients of the series $C(u)$ are fixed by $\iota$
which implies that $\iota : Z_\HC(\Ynls) \isoto Z_\HC(Y_{n,l}(\sigma_u))$ and that the elements
denoted $Z_1,\dots,Z_N$ in $\Ynls$ are sent to the elements with the same names in $Y_{n,l}(\sigma_u)$.
Furthermore it follows from the definition of $\iota$ that the generators of $Z_p(\Yns)$ are sent
bijectively to the generators of $Z_p(Y_n(\sigma_u))$, and we conclude that
$\iota : Z_p(\Ynls) \isoto Z_p(Y_{n,l}(\sigma_u))$. Now we can deduce (a) and (d) for $\Ynls$
from the same statements for $Y_{n,l}(\sigma_u)$.
\end{proof}

In the left-justified case, we saw in the proof above that $\gr Z_\HC(\Ynls)$ identifies with
$U(\g^e)^{G^e} \sub U(\g^e) \iso \gr \Ynls$.  It would be possible to prove this in general
by using a reduction modulo $p$ argument, but this fact is not required in the sequel.

For later use we record an immediate consequence of Theorem~\ref{T:centreY},
which describes a basis for $Z(\Ynls)$.  To do this
we use some notation introduced \S\ref{ss:truncate}.  For $\bu = (u_{i,j}^{(r)}) \in \bI_F$,
$\bt = (t_i^{(r)}) \in \bI_D$, $\bv = (v_{i,j}^{(r)}) \in \bI_E$ and
$\bw = (w_1,\dots,w_N) \in \{0,\dots,p-1\}^N$ we define
\begin{equation}
\label{e:PBWforthecentre}
(\dF^p)^\bu \dB^{\bt} (\dE^p)^\bv \dZ^\bw := \prod (\dF_{i,j}^{(r)}{}^p)^{u_{i,j}^{(r)}} \prod (\dB_i^{(rp)})^{t_i^{(r)}} \prod (\dE_{i,j}^{(r)}{}^p)^{v_{i,j}^{(r)}} \prod_{i=1}^{N} \dZ_i^{w_i}.
\end{equation}

\begin{Corollary}
\label{C:basisforthecentre}
A basis for $Z(\Ynls)$ is given by the ordered monomials
\begin{equation}
\label{e:basisforthecentre}
\{(\dF^p)^\bu \dB^{\bt} (\dE^p)^\bv \dZ^\bw \mid (\bu, \bt, \bv) \in \bI_F \times \bI_D \times \bI_E, \bw \in \{0,\dots,p-1\}^N\}
\end{equation}
\end{Corollary}

\subsection{Restricted (truncated) shifted Yangians}
\label{ss:restrictedshifted}
It is well-known that $U(\g)$
is a free module over its $p$-centre with a basis given by PBW monomials
in the standard basis of $\g$ in which every exponent is less than $p$; we refer to such
monomials as {\em $p$-restricted monomials}.
It follows that the restricted enveloping algebra $U^{[p]}(\g)$
is spanned by the image of the $p$-restricted monomials. Analogous statements hold for $\Yns$ and $\Ynls$, as we
now explain.

As explained at the end of \S\ref{ss:centreshiftedYangian}, we have that
$\Yns$ is a free $Z_p(\Yns)$-module with basis given by the $p$-restricted monomials in the
PBW generators of $\Yns$ given in \eqref{e:pbwgens}.
We define $Z_p(\Yns)_+$ to be the maximal ideal of $Z_p(\Yns)$
generated by the elements given in  \eqref{e:pgenerators}. Now we can define the
{\em restricted shifted Yangian} $\Ynps := \Yns /\Yns Z_p(\Yns)_+$.
The images in $\Ynps$ of the $p$-restricted
monomials in the PBW generators of $\Yns$ given in \eqref{e:pbwgens} form a basis of $\Ynps$.

As a consequence of Lemma~\ref{L:filtfree} and Theorem~\ref{T:centreY}(c), we see that $\Ynls$ is free
as an $Z_p(\Ynls)$-module.
To give a basis for this module we recall that from \eqref{e:YnlsPBWbasis} we
have the basis $\{\dF^\bu \dD^\bt \dE^\bv \mid (\bu,\bt,\bv) \in \bI_F \times \bI_D \times \bI_E\}$ of $\Ynls$.  We let $\bI_p$
be the set of all tuples $(\bu,\bt,\bv)$ where all entries of $\bu$, $\bt$ and $\bv$ are less than $p$.
Then the $p$-restricted monomials $\{\dF^\bu \dD^\bt \dE^\bv \mid (\bu,\bt,\bv) \in \bI_p\}$
form a basis of $\Ynls$ as a free $Z_p(\Ynls)$-module. We define $Z_p(\Ynls)_+$ to be the ideal
of $Z_p(\Yns)$ generated by the elements \eqref{e:pgenerators2} of $Z_p(\Ynls)$ and define the
{\em restricted truncated shifted Yangian} $\Ynlps := \Ynls /\Ynls Z_p(\Ynls)_+$.
Then a basis of $\Ynlps$ is given by
\begin{equation}
\label{e:PBWrestY}
\{\dF^\bu \dD^\bt \dE^\bv + \Ynls Z_p(\Ynls)_+  \mid (\bu,\bt,\bv) \in \bI_p\}
\end{equation}
In particular, we note that $\dim \Ynlps = p^{\dim \g^e}$.

We let $I_{n,l}^{[p]}$ be the ideal of $Y_n^{[p]}(\sigma)$ generated by $\{D_1^{(r)} + J_p(\Yns) \mid r > p_1 \}$.
Then using Theorem~\ref{T:centreY}(b) we can see
that there is a natural isomorphism
\begin{equation}
\Ynlps \isoto \Ynps / I_{n,l}^{[p]}.
\end{equation}

\subsection{The centre of $U(\g,e)$} \label{ss:centreW}
We use the description of $Z(\Ynls)$ given in Theorem~\ref{T:centreY}
along with the isomorphism $\phi : \Ynls \isoto U(\g,e)$ to
provide an explicit description of the centre $Z(\g,e)$ of $U(\g,e)$ as stated in Theorem~\ref{T:centreW} below.

We recall the map $\pr : U(\g) \to U(\p)$ is defined in \eqref{e:pr}
and define the {\em Harish-Chandra centre} $Z_\HC(\g,e)$ of $U(\g,e)$ to be the image
of $U(\g)^G$ under $\pr$. It is evident that $Z_\HC(\g,e)$ is invariant under the
twisted adjoint action of $M$, and
that these elements are central in $U(\g,e)$.
Our first objective is to show that the isomorphism $\phi : \Ynls \isoto U(\g,e)$ from \eqref{e:tildephi}
preserves the Harish--Chandra centres.

Recall that $Z_{\HC}(\Ynls)$ is generated by the coefficients of the polynomial $\dZ(u) \in \Ynls[u]$ defined in \S\ref{ss:centreTSY} whilst $U(\g)^G$ is generated by the coefficients of the Capelli determinant $Z^*(u) = \sum_{r=0}^N Z^{(r)} u^{N-r} \in U(\g)[u]$ given in \eqref{e:ZcentralUg}.
The following lemma relates these polynomials.

\begin{Lemma}
\label{L:HClemmaUg2}
We have the following equality in $U(\g,e)[u]$
\begin{equation}
\label{e:HCpowerseriesrelation}
\pr(Z^*(u)) = \phi(\dZ(u)).
\end{equation}
\end{Lemma}

\begin{proof}
Recall that $\pr_\Z : U(\g_\Z) \to U(\p_\Z)$ is given in \eqref{e:prZ}.
If we view $Z^*(u)$ as a polynomial with coefficients in $U(\g_\Z)$
then $\pr_\Z(Z^*(u))$ is a polynomial with coefficients in $U(\p_\Z)$.
Using Lemma~\ref{L:integraltruncatedPBW}(b) we view $Y_{n,l}^\Z(\sigma)$ as a subalgebra
of $Y_{n,l}^\C(\sigma)$, and thus view $\dZ(u)$ as a polynomial with coefficients
in $Y_{n,l}^\Z(\sigma)$. Recalling the map $\phi_\Z$ from \eqref{e:phiZ} we obtain two
polynomials $\pr_\Z(Z^*(u))$ and $\phi_\Z(\dZ(u))$ with coefficients in $U(\p_\Z)$.
Using the natural inclusion $U(\p_\Z) \into U(\p_\C)$, \cite[Lemma 3.7]{BKrep}
implies that the equality $\pr_\Z (Z^*(u)) = \phi_\Z(\dZ(u))$ holds in $U(\p_\Z)[u]$.
Now, by taking the image of this equality under the natural map
$U(\p_\Z)[u] \to U(\p_\Z)[u]\otimes_\Z \kk \cong U(\p)[u]$, we
obtain \eqref{e:HCpowerseriesrelation}.
\end{proof}

We introduce the notation $Z_r := \pr(Z^{(r)}) \in U(\g,e)$ for $r=1,\dots,N$; by the
previous lemma we have that $Z_r = \phi(\dZ_r)$ too.

\begin{Corollary} \label{C:HClineup}
We have $\phi : Z_\HC(\Ynls) \isoto Z_\HC(\g,e)$.
\end{Corollary}

\begin{proof}
In \S\ref{ss:Znilpcent} we demonstrated that $U(\g)^G$ is generated by the coefficients of
$Z^*(u)$, and it follows that $Z_\HC(\g,e)$ is generated by the coefficients of $\pr Z^*(u)$, i.e.\ by $Z_1,\dots,Z_n$. Now
Lemma \ref{L:HClemmaUg2} implies that the generators $\dZ_1,\dots,\dZ_N$ of $Z_\HC(\Ynls)$ are sent bijectively
to those of $Z_\HC(\g,e)$.
\end{proof}

The {\em $p$-centre of $U(\g,e)$} is defined to be
$$
Z_p(\g,e) := Z_p(\p)^{\tw(M)} \sub U(\g,e).
$$
In the general setting of finite $W$-algebras associated to reductive groups, this subalgebra
was studied in some detail in \cite[Section~8]{GTmod}. Using the explicit formulas for the generators
 \eqref{e:Wgens} of $U(\g,e)$ given in \S\ref{ss:finiteW} we now introduce an
explicit generating set for $Z_p(\g,e)$.
Recall that the Kazhdan filtration of $U(\p)$ and $U(\g,e)$ was discussed at the end of \S\ref{ss:finiteW};
in particular, we identify $\gr' U(\g,e) \iso S(\p)^{\tw M}$.  Also we remind the reader that $\xi_\p : S(\p)^{(1)} \to Z_p(\p)$
is defined in \textsection\ref{ss:alggps}

\begin{Lemma}
\label{L:ZpUgegens}
$ $
\begin{enumerate}
\item[(a)] $Z_p(\g,e)$ is a polynomial algebra of rank $\dim \g^e$ generated by
\begin{align}
\label{e:Zpgegenerators}
\begin{array}{c}
 \{\xi_\p(\gr'_r D_i^{(r)}) \mid (i,r) \in \bJ_D\} \cup \{\xi_\p(\gr_r' E_{i,j}^{(r)}) \mid (i,j,r) \in \bJ_E\}\\   \cup\, \{\xi_\p(\gr_r' F_{i,j}^{(r)}) \mid (i,j,r) \in \bJ_F\}.\end{array}
\end{align}
\item[(b)] Explicitly we have
\begin{align}
\label{e:explicitDpcentre}
\begin{array}{l}\xi_\p(\gr_r' D_i^{(r)}) = \vspace{6pt}\\ \ \ \ \ \  \sum_{s=1}^r (-1)^{r-s} \sum_{\substack{i_1,\dots,i_s \\ j_1,\dots,j_s}} (-1)^{|\{t=1,\dots,s-1 \mid \row(j_t) \le i-1\}|}
(e_{i_1, j_1}^p - e_{i_1, j_1}^{[p]}) \cdots (e_{i_s, j_s}^p - e_{i_s, j_s}^{[p]}) \end{array}
\end{align}
where the sum is taken over the index set described in \eqref{e:Dir}.
\end{enumerate}
\end{Lemma}

\begin{proof}
As remarked at the end of \S\ref{ss:finiteW}, $S(\p)^{\tw(M)}$ is a polynomial algebra of rank $\dim \g^e$ generated by
$\{\gr'_r D_i^{(r)}, \gr_r' E_{i,j}^{(r)}, \gr_r' F_{i,j}^{(r)}\}$.
Using \cite[Lemma~7.6]{GTmod} we note that the restriction of $\pr : U(\g) \to U(\p)$ to $Z_p(\g)$ is the projection
$Z_p(\g) \to Z_p(\p)$ along the decomposition $Z_p(\g) = Z_p(\g)\{x^p -x^{[p]} - \chi(x)^p \mid x \in \m\} \oplus Z_p(\p)$.
It follows that $\xi_\p : S(\p)^{(1)} \to Z_p(\p)$ is equivariant for the twisted action of $M$, so we can deduce (a).

Part (b) now follows easily from (a), because the formula for
$\gr_r' D_i^{(r)}$ is obtained from \eqref{e:Dir} by replacing each occurrence of $\te_{i_l,j_l}$ with $e_{i_l, j_l}$.
\end{proof}

Using the explicit formulas for $E_i^{(r)}$ and $F_i^{(r)}$ given in \cite[Section~4]{GTmin} we can give precise formulas for the
generators $\xi_\p(\gr_r' E_i^{(r)})$ and $\xi_\p(\gr_r' F_i^{(r)})$ analogous
to that given for $\xi_\p(\gr_r' D_i^{(r)})$. In principle, it is also
possible, though more complicated, to provide expressions for the generators
$\xi_\p(\gr_r' E_{i,j}^{(r)})$ and $\xi_\p(\gr_r' F_{i,j}^{(r)})$ when $i < j+1$.

We are now ready to prove our main result regarding the centre of $U(\g,e)$.  For the statement
of this theorem, we consider the intersection $Z_{\HC,p}(\g,e) := Z_{\HC}(\g,e) \cap Z_p(\g,e)$.
It is a direct consequence of the definitions that this intersection is equal to
$\pr(Z_p(\g)^G)$.

\begin{samepage}
\begin{Theorem} \label{T:centreW} $ $
\begin{itemize}
\item[(a)] The centre $Z(\g,e)$ of $U(\g,e)$ is free of rank $p^N$ over $Z_p(\g,e)$ with basis
$$
\{Z_1^{k_1} \cdots Z_N^{k_N} \mid 0 \le k_i < p\}.
$$
\item[(b)] We have a tensor product decomposition
$$
Z(\g,e) = Z_p(\g,e) \otimes_{Z_{\HC,p}(\g,e)} Z_{\HC}(\g,e).
$$
\end{itemize}
\end{Theorem}
\end{samepage}

\begin{proof}
By Lemma~\ref{L:ZpUgegens} and the formulas given in \eqref{e:Wgensloop}, we
have $\gr Z_p(\g,e) = Z_p(\g^e)$.  Further, by Theorem~\ref{T:centreY}(b) we have $\gr Z_p(\Ynls) = Z_p(\g^e)$.
The isomorphism $\phi : \Ynls \isoto U(\g,e)$ is filtered
with respect to the loop filtration by Theorem~\ref{T:isoandPBW}, and sends
$\dZ_r \in Z_{\HC}(\Ynls)$
to $Z_r \in Z_{\HC}(\g,e)$.  Now (a) follows from Theorem~\ref{T:centreY}(c).

To prove (b), we apply Lemma~\ref{L:tensorprodpresentation}, with $B = Z_p(\g,e)$ and $C = Z_{\HC}(\g,e)$,
and the set of generators $\{c_1,\dots,c_m\} = \{Z_1^{k_1} \cdots Z_N^{k_N} \mid 0 \le k_i < p\}$.
The first condition that we need to verify is given in (a), so we are left to verify that
$Z_{\HC}(\g,e)$ is generated as a $Z_{\HC,p}(\g,e)$ by $\{Z_1^{k_1} \cdots Z_N^{k_N} \mid 0 \le k_i < p\}$.
As explained after Lemma~\ref{L:Ucentralisercentre}, in the case $e =0$, we have $z_r = Z^{(r)}$.
Thus from this lemma we obtain that $Z_{\HC}(\g)$ is generated as a $Z_p(\g)^G$-module by
$\{(Z^{(1)})^{k_1} \cdots (Z^{(N)})^{k_N} \mid 0 \le k_i < p\}$.  Since $\pr$ sends
$Z^{(r)}$ to $Z_r$ by Lemma~\ref{L:HClemmaUg2} we deduce the desired result.
\end{proof}

We set up some notation for a basis of $Z(\g,e)$. For $(\bu, \bt, \bv) \in \bI_F \times \bI_D \times \bI_E$
and $\bw \in \{0,1,\dots,p-1\}^N$, we define
\begin{align*}
\begin{array}{ll} \xi_\p(\gr' F)^\bu \xi_\p(\gr' D)^\bt \xi_\p(\gr' E)^\bv Z^\bw := \\
\ \ \ \ \ \ \ \ \ \ \ \ \  \prod \xi_\p(\gr' F_{i,j}^{(r)})^{u_{i,j}^{(r)}} \prod \xi_\p(\gr' D_i^{(r)})^{t_i^{(r)}} \prod \xi_\p(\gr' E_{i,j}^{(r)})^{v_{i,j}^{(r)}} \prod Z_i^{w_i}.
 \end{array}
\end{align*}
Then
the ordered monomials
\begin{equation} \label{e:Zgebasis}
\{\xi_\p(\gr' F)^\bu \xi_\p(\gr' D)^\bt \xi_\p(\gr' E)^\bv Z^\bw \mid (\bu, \bt, \bv, \bw) \in \bI_F \times \bI_D \times \bI_E \times \{0,\dots,p-1\}^N\}
\end{equation}
form a basis for $Z(\g,e)$.

\subsection{Restricted finite $W$-algebras}
\label{ss:restrictedW}
We move on to recall the definition of the restricted $W$-algebra $U^{[p]}(\g,e)$.
We write $Z_p(\p)_+$ for the ideal of $Z_p(\p)$ generated by
$\{x^p - x^{[p]}  \mid x \in \p\}$, so the restricted enveloping algebra of $\p$
is $U^{[p]}(\p) = U(\p)/U(\p) Z_p(\p)_+$.
Then the {\em restricted $W$-algebra} is defined as
\begin{equation*}
U^{[p]}(\g,e) := U(\g,e) / (U(\g,e) \cap U(\p)Z_p(\p)_+).
\end{equation*}
Since, the kernel of the restriction of the projection $U(\p) \onto U_0(\p)$
to $U(\g,e)$ is $U(\g,e) \cap U(\p)Z_p(\p)_+$, we can identify
$U^{[p]}(\g,e)$ with the image of $U(\g,e)$ in $U^{[p]}(\p)$.

By \cite[Theorem~8.4]{GTmod}, we have that $U(\g,e)$ is free of rank $p^{\dim \g^e}$ over
$Z_p(\g,e)$, and thus that $\dim U^{[p]}(\g,e) = p^{\dim \g^e}$.
We note that each of the elements in \eqref{e:Zpgegenerators} lies in
$U(\g,e) \cap Z_p(\p)_+$, and we let $Z_p(\g,e)_+$ be the ideal of $Z_p(\g,e)$ generated
by these elements.  By Lemma~\ref{L:ZpUgegens}(a), we have that $Z_p(\g,e)_+$ is a maximal
ideal of $Z_p(\g,e)$, and it follows that $Z_p(\g,e)_+ = U(\g,e) \cap Z_p(\p)_+$.
By using the formulas given in \eqref{e:Wgensloop},
and a filtration argument we see that $U(\g,e)/U(\g,e)Z_p(\g,e)_+$ is spanned by the $p$-restricted
monomials in the elements in \eqref{e:Zpgegenerators}.
Hence, we see that $U(\g,e) \cap U(\p)Z_p(\p)_+ = U(\g,e)Z_p(\g,e)_+$, and obtain the basis
\begin{equation}
\label{e:PBWrestW}
\{F^\bu D^\bt E^\bv + U(\g,e)Z_p(\g,e)_+ \mid (\bu,\bt,\bv) \in \bI_p\}
\end{equation}
of $U^{[p]}(\g,e)$.

\section{Highest weight modules for $Y_{n,l}(\sigma)$ and $U(\g,e)$} \label{S:hwY}

For our proof of Theorem~\ref{T:restricted}, we require some results about
highest weight vectors in modules for $Y_{n,l}(\sigma)$ and $U(\g,e)$.  In this section
we cover the required material, with the key results being Lemmas~\ref{L:Birpaction}
and \ref{L:DonUh}.  We continue to use the notation from Sections \ref{S:YandW} and \ref{S:cent}.

\subsection{Torus actions}
\label{ss:torusactions}
Before discussing highest weight theory we have to introduce the
underlying torus actions.

Let $T_n$ be the maximal torus of $\GL_n(\kk)$ of diagonal matrices.
We write $\{\epsilon_1,\dots,\epsilon_n\}$ for the standard basis of the
character group $X^*(T_n)$ of $T_n$, i.e.\ $\epsilon_i : T_n \to \kk^\times$
is defined by $\epsilon_i(\diag(t_1,\dots,t_n)) = t_i$.  The positive
weights in $X^*(T_n)$ are
$X^*_+(T_n) = \{\sum_{i=1}^n a_i \epsilon_i \in X^*(T) \mid a_i \in \Z, \, a_i \ge a_{i+1} \text{ for all } i \text{ and } a_1 > a_n\}$.

Now let $T$ be the maximal torus of $G$ of diagonal matrices, and let $T^e$ be the centralizer
of $e$ in $T$.
We can describe $T^e$ explicitly in terms of certain cocharacters.
Define $\tau_1,\dots,\tau_n : \kk^\times \to T$, where
$\tau_i(t)$ is the diagonal matrix with $j$th entry equal to $t$ if
$\row(j) = i$
and entry $1$ otherwise.  Then we have $T^e = \{ \prod_{i=1}^n \tau_i(t_i) \mid t_i \in \kk^\times\}$.
Thus we have an isomorphism
\begin{equation} \label{e:Tniso}
T_n \isoto T^e
\end{equation}
which sends $\diag(t_1,\dots,t_n)$
to $\prod_{i=1}^n \tau_i(t_i)$.  From now on we use the above isomorphism
to identify $T_e$ with $T_n$.
It is a straightforward to see that
the basis element $c_{i,j}^{(r)}$ of $\g^e$ is a $T_n$-weight vector with
weight $\epsilon_i - \epsilon_j$.

We note that the adjoint action of $T^e$ on $U(\g)$ restricts to
an adjoint action on $U(\g,e)$, so we have an
action of $T_n$ on $U(\g,e)$.
By inspection of the formula for $D_i^{(r)}$ in \eqref{e:Dir}, we see that
it is fixed by $T_n$.  Similarly, by considering the formula
for $E_i^{(r)}$ given in \cite[Section~4]{GTmin}, we see that
$E_i^{(r)}$ is a $T_n$-weight vector with weight $\epsilon_i - \epsilon_{i+1}$; and
then deduce, using \eqref{e:Eij} that $E_{i,j}^{(r)}$ has $T_n$-weight $\epsilon_i - \epsilon_j$.
Similarly, we see that $F_{i,j}^{(r)}$ has $T_n$-weight $\epsilon_j - \epsilon_i$.

Further, we note that the action on $T_n$ on $U(\g,e)$ is filtered for
the loop filtration, so there is an action
of $T_n$ on $\gr U(\g,e)$.  Under the identification $\gr U(\g,e) \iso U(\g^e)$ given
by $S_{-\eta}$ in \eqref{e:S-eta}, this action coincides with the natural action
of $T_n \iso T^e$ on $U(\g^e)$.

By considering the relations for
$\Yns$ given in \cite[Theorem~4.15]{BT} and the definitions
of $E_{i,j}^{(r)}$ and $F_{i,j}^{(r)}$ given in \eqref{e:Eij},
we see that there is an action of $T_n$ on $\Yns$ by algebra automorphisms, such that $D_i^{(r)}$ is fixed by $T_n$, the weight
of $E_{i,j}^{(r)}$ is $\epsilon_i - \epsilon_j$, and the weight of $F_{i,j}^{(r)}$ is $\epsilon_j - \epsilon_i$.
Further, this action of $T_n$ is filtered for the loop filtration, and through the isomorphism
$\psi: U(\cns) \isoto \gr \Yns$  in \eqref{e:isogrTSY} it corresponds to the natural action of $T_n$ on $U(\cns)$.

We note that the ideal $I_{n,l}$ is $T_n$-stable, so that there is an induced action of $T_n$ on $\Ynls$.
From the description of the action of $T_n$ on $\Yns$ and on $U(\g,e)$ above,
we see that the isomorphism $\phi : \Ynls \isoto U(\g,e)$ in \eqref{e:phi} is $T_n$-equivariant.

\subsection{Highest weight modules for $\Ynls$} \label{ss:hwY}
For our proof of Theorem~\ref{T:restricted}, we require some theory of
highest weight modules for $\Ynls$.  We outline what we need below, much of which
is a modular analogue of some results in \cite[Chapter 6]{BKrep}, though here we take
a more elementary approach to some of the results we require.
The key result in this subsection is Lemma~\ref{L:Birpaction}, which tells us how the
elements $\dB_i^{(rp)}$ act on highest weight vectors.

We recall that a PBW basis $\{\dF^\bu \dD^\bt \dE^\bv \mid (\bu,\bt,\bv) \in \bI_F \times \bI_D \times \bI_E\}$ of
$\Ynls$ is given in \eqref{e:YnlsPBWbasis}.
In the discussion below we also require the ordered sets $\bJ_F$, $\bJ_D$ and $\bJ_E$, which are defined
before \eqref{e:YnlsPBWbasis}, and used to fix the order in the PBW monomials.
Since each $F_{i,j}^{(r)}$, $D_i^{(r)}$ and $E_{i,j}^{(r)}$
is a $T_n$-weight vector we see that the elements of the above PBW basis of $\Ynls$
are also $T_n$-weights.
In order to define Verma modules for $\Ynls$ we fix $\ba = (a_i^{(r)} \mid 1 \le i \le n, 1 \le r \le p_i) \in \kk^N$.
We use this tuple to modify the basis given in \eqref{e:YnlsPBWbasis} by setting
\begin{align*}
\dF^\bu (\dD - \ba)^\bt \dE^\bv = \prod_{(i,j,r) \in \bJ_{F}} (\dF_{i,j}^{(r)})^{u_{i,j}^{(r)}} \prod_{(i,r) \in \bJ_{D}} (\dD_i^{(r)} - a_{i}^{(r)})^{t_i^{(r)}}
\prod_{(i,j,r) \in \bJ_{E}} (\dE_{i,j}^{(r)})^{v_{i,j}^{(r)}}.
\end{align*}
Then we see that
\begin{align}
\label{e:PBWmonomialsfromA}
\{\dF^\bu (\dD - \ba)^\bt \dE^\bv \mid (\bu,\bt,\bv) \in \bI_F \times \bI_D \times \bI_E\}.
\end{align}
forms a basis for $\Ynls$.
Also we note that these basis elements are $T_n$-weight vectors, and that the $T_n$-weight of
$\dF^\bu (\dD - \ba)^\bt \dE^\bv$ is the same as that of $\dF^\bu \dD^\bt \dE^\bv$.
We define $\cM(\ba)$ to be the set of monomials \eqref{e:PBWmonomialsfromA}
for which $\bt \ne 0$ or $\bv \ne 0$, and write $I(\ba)$ for the
subspace of $\Ynls$, which has these as a basis.

\begin{Lemma}
\label{L:definingVermas}
Let $\ba = (a_i^{(r)} \mid 1 \le i \le n, 1 \le r \le p_i) \in \kk^N$, and define $I(\ba)$ as
above. Then:
\begin{enumerate}
\item[(a)] any $T_n$-weight vector in $\Ynls$ with weight
in $X_+^*(T_n)$ lies in  $I(\ba)$; and
\item[(b)] $I(\ba)$ is a left ideal of $\Ynls$.
\end{enumerate}
\end{Lemma}

\begin{proof}
For a monomial $\dF^\bu (\dD - \ba)^\bt \dE^\bv$
to have a positive weight, it must have $\bv \ne 0$, and thus lies in $I(\ba)$.
From this we can deduce (a) as these monomials give a basis of $\Ynls$.

For the proof of (b), we require another filtration of $\Ynls$, known
as the {\em canonical filtration}.
First we recall that the canonical filtration is defined on $\Yns$
by placing $E_{i,j}^{(r)}, D_i^{(r)}, F_{i,j}^{(r)}$ in filtered degree $r$,
then we get the induced filtration on $\Ynls$.
We write $\gr' \Yns$ and $\gr' \Ynls$ for the associated graded algebras
for the canonical filtrations.
As is remarked in \cite[\S4.2]{BT}, $\gr' \Yns$ is commutative, and thus
$\gr' \Ynls$ is also commutative.

Let $X' = \dF^{\bu'} (\dD-a)^{\bt'} \dE^{\bv'}$ be in the basis given in
\eqref{e:PBWmonomialsfromA} and $X = \dF^\bu (\dD-a)^\bt \dE^\bv \in \cM(\ba)$.
We write $\deg'(X'X)$ for the canonical degree of $X'X$ and proceed to prove that $X'X \in I(\ba)$ by
induction on $\deg'(X'X)$.
It is clear that (b) will follow immediately from this.

If $\deg'(X'X) = 0$, then $X', X \in \kk$ (and in fact $X = 0$) so the claim holds.
So we suppose that $\deg'(X'X) > 0$.

For our fixed value of $\deg'(X'X)$, we see that
we can reduce to the case where $\bu = 0$, by writing
$\dF^{\bu'} \dD^{\bt'} \dE^{\bv'} \dF^\bu$ as $\dF^{\bu'+\bu} \dD^{\bt'} \dE^{\bv'}$ plus a sum
of the PBW monomials in the basis given in \eqref{e:PBWmonomialsfromA} of strictly lower
canonical degree; here we use that $\gr' \Ynls$ is commutative.
Thus we assume that $X = (\dD-a)^\bt \dE^\bv$.
We define the length $\ell(X')$ to be the sum of the entries of all three tuples $\bu', \bt', \bv'$,
and now work by induction on $\ell(X')$, under the assumption that $X$ is of the form
$(\dD-a)^\bt \dE^\bv \in \cM(\ba)$.

If $\ell(X') = 0$, then
$X' = 1$, and trivially $X'X \in I(\ba)$.  So we assume that $\ell(X') > 0$.

Suppose that $\bv' = \bt' = 0$, then
we see that $X'X = \dF^{\bu'} (\dD - \ba)^{\bt} \dE^{\bv} \in \cM(\ba)$.

Next suppose that $\bv' = 0$ and $\bt' \ne 0$.
Let $(i_0,r_0)$ be largest with respect to our fixed order on $\bJ_D$ such that
$t_{i_0}'^{(r_0)} \ne 0$.  Define $\bs = (s_i^{(r)} \mid 1\leq i \leq n, 1 \leq r \leq p_i) \in \bI_D$ by $s_i^{(r)} = \delta_{i,i_0} \delta_{r, r_0}$.
We see that $X'X = (\dF^{\bu'} (\dD-\ba)^{\bt' - \bs}) ((\dD-\ba)^{\bt + \bs} \dE^\bv)$,
as the $D_i^{(r)}$ all commute with each other.
Since $\ell(\dF^{\bu'} (\dD-\ba)^{\bt' - \bs}) < \ell(X')$ and $(\dD-\ba)^{\bt + \bs} \dE^\bv \in \cM(\ba)$ is of the required form,
we conclude that $X'X \in I(\ba)$ by induction
on $\ell(X')$.

Last we consider the case $\bv' \neq 0$.  Let $(i_0,j_0,r_0)$ be largest with respect to our fixed order on $\bJ_E$ such that
$v_{i_0,j_0}'^{(r_0)} \ne 0$. Define $\bq = (q_{i,j}^{(r)} \mid 1\leq i< j  \leq n, 1 \leq r \leq p_i) \in \bI_E$ by
$q_{i,j}^{(r)} = \delta_{i,i_0} \delta_{j,j_0} \delta_{r, r_0}$.
Also write $\dE^\bv = E^{\bv_<} \dE^{\bv_\ge}$, where $\dE^{\bv_<}$ is the
submonomial of $\dE^\bv$ consisting of the $\dE_{i,j}^{(r)}$ for $(i,j,r)$ up
to $(i_0,j_0,r_0)$ in our fixed order of $\bJ_E$, and $\dE^{\bv_\ge}$ is the remaining submonomial.
We have
\begin{align*}
X'X &= (\dF^{\bu'} (\dD-\ba)^{\bt'} \dE^{\bv'}) ((\dD-\ba)^\bt \dE^\bv) \\
&= (\dF^{\bu'} (\dD-\ba)^{\bt'} \dE^{\bv'-\bq}) \left(
(\dD-\ba)^\bt \dE^{\bv+\bq}
+
[\dE_{i_0,j_0}^{(r_0)},(\dD-a)^\bt \dE^{\bv_<}] \dE^{\bv_{\ge}} \right). \\
\end{align*}
The first term $(\dF^{\bu'} (\dD-\ba)^{\bt'} \dE^{\bv'-\bq}) ((\dD-\ba)^\bt \dE^{\bv+\bq})$ above satisfies
$\ell(\dF^{\bu'} (\dD-\ba)^{\bt'} \dE^{\bv'-\bq}) < \ell(X')$ and $(\dD-a)^\bt \dE^{\bv+\bq} \in \cM(\ba)$ is of the required form.
So we conclude that this term lies in $I(\ba)$ by induction on $\ell(X')$.  We are left to consider the term
$Y = (\dF^{\bu'} (\dD-\ba)^{\bt'} \dE^{\bv'-\bq})([\dE_{i_0,j_0}^{(r_0)},(\dD-a)^\bt \dE^{\bv_<}] \dE^{\bv_{\ge}})$.
As the associated graded algebra of $\Ynls$ for the canonical filtration
is commutative, we have that $\deg' Y < \deg' X'X$.  Next we see that $[\dE_{i_0,j_0}^{(r_0)},(\dD-a)^\bt \dE^{\bv_<}] \dE^{\bv_{\ge}}$
has positive $T_n$-weight, so lies in $I(\ba)$ by (a).  Therefore, it can be rewritten as a linear combination of monomials
in $\cM(\ba)$.  Let $\tilde X$ be a monomial from $\cM(\ba)$ occurring in this sum.  Then we know that
$(\dF^{\bu'} (\dD-\ba)^{\bt'} \dE^{\bv'-\bs}) \tilde X \in I(\ba)$, by induction on $\deg'(X'X)$.  Putting this all together we obtain that
$X'X \in I(\ba)$ as required, which completes the double induction.
\end{proof}

Now we define the {\em Verma module}
\begin{equation*}
M(\ba) : = \Ynls/I(\ba).
\end{equation*}
From the PBW theorem, it is clear that a basis of $M(\ba)$ is given by $\{F^\bu + I(\ba) \mid \bu \in \bI_F\}$.
It follows immediately from the definition of $I(\ba)$ that $\dD_i^{(r)}$ acts on $1 + I(\ba)$
as $a_i^{(r)}$ for all $(i,r) \in \bJ_D$ and that
$\dE_{i,j}^{(r)}$ annihilates $1 + I(\ba)$ for all $(i,j,r) \in \bJ_E$.
In fact, something much stronger is true.

\begin{Lemma}
\label{L:thingsvanish}
The following elements of $\Ynls$ annihilate $1 + I(\ba) \in M(\ba)$:
\begin{enumerate}
\item[(a)] $\dE_{i,j}^{(r)}$ for all $1 \le i < j \le n$, $r > s_{i,j}$.
\item[(b)] $\dD_i^{(r)}$ for all $1 \le  i \le n$, $r > p_i$;
\end{enumerate}
\end{Lemma}

\begin{proof}
Since $T_n$ acts on $E_{i,j}^{(r)}$ with a positive weight for all $1 \le i < j \le n$, $r > s_{i,j}$,
part (a) follows from Lemma~\ref{L:definingVermas}(a).

Using \eqref{e:Tij} we calculate that
\begin{equation*}
\dT_{i,i}^{(r)} = \dD_i^{(r)} + \sum_{k=1}^i \sum_{\substack{a > s_{i,k}, b \ge 0 , c > s_{k,j}\\ a + b + c = r}} \dF_{k,i}^{(a)} \dD_k^{(b)} \dE_{k,j}^{(c)}.
\end{equation*}
By Corollary \ref{C:Tij0} we know that $\dT_{i,i}^{(r)} = 0$ for $r > p_i$.  Also by (a), we know that
each $\dE_{k,j}^{(c)}$ on the righthand side of the above equation annihilates $1+I(\ba)$.  Hence,
we deduce that $\dD_i^{(r)}$ also annihilates $1+I(\ba)$ for $r > p_i$.
\end{proof}

Let $M$ be a $\Ynls$-module and let $v_+ \in M$.  We say that $v_+$ is a highest
weight vector of weight $\ba$ if $I(\ba)$ annihilates $v_+$. We say that $M$ is a highest weight module of weight $\ba$ if $M$ is generated by some highest weight vector of weight $\ba$.
The Verma modules $\{M(\ba) \mid \ba = (a_i^{(r)})_{1\le i \le n}^{1\le r \le p_i} \in \kk^N\}$ are defined to be the universal highest weight
modules. Thus if $v_+ \in M$ is a highest weight vector of weight $\ba$, then there is a unique
map $M(\ba) \to M$ sending $1 + I(\ba)$ to $v_+$.

It is helpful for us to relabel the Verma modules, following the approach of \cite[Section~6.1]{BKrep}.  Suppose we have a highest weight vector $v_+$
with weight $\ba$ in some $\Ynls$-module.  Then by Lemma~\ref{L:thingsvanish} we know that
$$
u^{p_i} \dD_i(u) v_+ = (u^{p_i} + a_i^{(1)}u^{p_i-1} + \dots + a_i^{(p_i-1)} u + a_i^{(p_i)}) v_+.
$$
By factorising and introducing a shift, we have that
\begin{equation} \label{e:Dact}
u^{p_i} \dD_i(u) v_+ = (u+(i-1)+a_{i,1})(u+(i-1)+a_{i,2})\dots(u+(i-1)+a_{i,p_i}).
\end{equation}
These are the formulas given in \cite[(6.1)--(6.3)]{BKrep}.

We let $A$ be the $\pi$-tableau with entries $\{a_{i,j} \mid j = 1,\dots,p_i\}$ on the $i$th row,
and note that $A$ is only defined up to row equivalence.
We denote the row equivalence class of $A$ by $\bar A$, and from now on we refer
to $\bar A$ as the weight of $v_+$, rather than $\ba$.
This allows an alternative parametrization of the Verma modules, where
we write $M(\bar A)$ instead of $M(\ba)$; we use the notation $v_{A,+}$ for the highest weight vector of $M(\bar A)$.

We define $\Ynls^0$ to be the (commutative) subalgebra
of $\Ynls$ generated by $\{\dD_i^{(r)} \mid 1\leq i \leq n, 0 < r \leq p_i\}$
and note that $\Ynls^0$ is in fact a polynomials algebra on these
generators.
The next lemma is a direct consequence of the Nullstellensatz,
but we record it for convenience of reference.

\begin{Lemma} \label{L:Dacting}
Let $d \in \Ynls^0$.  Then $dv_{A,+} = 0$ for all
$A \in \Tab(\pi)$ if and only if $d = 0$.
\end{Lemma}

A useful observation for us gives the action of the generators $\dZ_r$ of $Z_\HC(\Ynls)$ on
a highest weight vector of weight $A \in \Tab(\pi)$.  We
calculate
\begin{equation}
\begin{aligned}
\label{e:Zact}
\dZ(u)v_{A,+} &= u^{p_1} (u-1)^{p_2} \cdots (u-(n-1))^{p_n}\dD_1(u) \dD_2(u-1) \dots \dD_n(u-(n-1))v_{A,+} \\
&= u^{p_1} \dD_1(u) (u-1)^{p_2} \dD_2(u-1) \dots (u-(n-1))^{p_n} \dD_n(u-(n-1))v_{A,+} \\
&= (u+a_{1,1}) \dots (u+a_{1,p_1})(u+a_{2,1}) \dots (u+a_{n,p_n})v_{A,+}.
\end{aligned}
\end{equation}
Therefore, we see that $\dZ_r$ acts as $e_r(a_{i,j} \mid 1 \le i \le n, 1 \le j \le p_i)$,
where we recall that $e_r$ denotes the $r$th elementary symmetric polynomial.

Now we want to calculate the scalar by which $B_i^{(rp)}$ acts on
the highest weight vector $v_{A_,+}$.

\begin{Lemma} \label{L:Birpaction}
Let $A \in \Tab(\pi)$, let $1 \le k \le n$ and $1 \le r \le p_i$. Set $s := s(r) = r + \lfloor \frac{r-1}{p-1}\rfloor$ and
let $\bD_{i,r}$ be the set of all sequences $\bd = (d_0,d_1,d_2,\dots,d_s)$ of non-negative integers such that
$\sum_{j \ge 0} d_j = p_i$ and $rp = \sum_{j \ge 1} d_j (jp-j+1)$.
Then
\begin{align}
\label{e:Bacts}
\dB_i^{(rp)} v_{A,+} = \sum_{\bd \in \bD_{i,r}} \frac{(\sum_{j \ge 1} d_j)!}{\prod_{j \ge 1} d_j!} e_{\sum_{j \ge 1} d_j}(a_{i,1}^p-a_{i,1},\dots,a_{i,p_i}^p-a_{i,p_i})
\end{align}
Furthermore, there exist elements $\hat B_i^{(rp)} \in Z_p(\Ynls)$ related to $\dB_i^{(rp)}$ by a unitriangular change of variables, such that
$$
\hat B_i^{(rp)} v_{A,+} = e_r(a_{i,1}^p-a_{i,1},\dots,a_{i,p_i}^p-a_{i,p_i}).
$$
Finally, if $p > r$, then $\hat B_i^{(rp)} = \dot B_i^{(rp)}$.
\end{Lemma}

\begin{proof}
We consider $\left(\prod_{j=0}^{p-1} (u-j)^{p_i} \right) \dB_i(u) v_{A,+}$ and calculate
\begin{align*}
\left(\prod_{j=0}^{p-1} (u-j)^{p_i} \right) \dB_i(u) v_{A,+} &= \left(\prod_{j=0}^{p-1} (u-j)^{p_i} \dD_i(u-j)\right) v_{A,+} \\
&= \prod_{j=0}^{p-1} \prod_{k=1}^{p_i} (u-j+a_{i,k}) v_{A,+} \\
&= \prod_{k=1}^{p_i}\prod_{j=0}^{p-1}  (u+a_{i,k}-j) v_{A,+} \\
&= \prod_{k=1}^{p_i} \left((u+a_{i,k})^p-(u+a_{i,k})\right) v_{A,+} \\
&= \prod_{k=1}^{p_i} (u^p-u+(a_{i,k}^p-a_{i,k})) v_{A,+}.
\end{align*}
We explain some of the steps in the above calculation.  The first equality
just uses the definition of $\dB_i(u)$. To go from the first line to the second
we use the definition of the action of $\dD_i(u)$ in \eqref{e:Dact}.  Then to go from the third line
to the fourth we use \eqref{e:Fppolya}.

Also we have that $\prod_{j=0}^{p-1} (u-j)^{p_i} = (u^p-u)^{p_i}$ by \eqref{e:Fppoly}, so we obtain
\begin{align*}
(u^p-u)^{p_i} \dB_i(u) v_{A,+} &= \prod_{k=1}^{p_i} (u^p-u+(a_{i,k}^p-a_{i,k})) v_{A,+} \\
&= (u^p-u)^{p_i} \prod_{k=1}^{p_i} (1+(a_{i,k}^p-a_{i,k})(u^p-u)^{-1}) v_{A,+}.
\end{align*}
Thus
\begin{align} \label{e:Bprod}
\dB_i(u) v_{A,+} &= \prod_{k=1}^{p_i} (1+(a_{i,k}^p-a_{i,k})u^{-p}(1-u^{-(p-1)})^{-1}) v_{A,+} \nonumber \\
&= \prod_{k=1}^{p_i} \left(1+(a_{i,k}^p-a_{i,k}) u^{-p} \sum_{j \ge 0} u^{-j(p-1)}\right) v_{A,+} \nonumber \\
&= \prod_{k=1}^{p_i}  \left(1+(a_{i,k}^p-a_{i,k}) (u^{-p} +u^{-(2p-1)} + u^{-(3p-2)} + \dots \right) v_{A,+}.
\end{align}
The action of $\dB_i^{(rp)}$ on $v_+$ is determined by the coefficient of $u^{-rp}$ in the above expression.

Let $\bd \in \bD_{i,r}$ and let
$$
\binom{p_i}{d_0, d_1, \dots,d_s} = \frac{p_i!}{\prod_{j \ge 0} d_j!}.
$$
be the multinomial coefficient. By choosing the summand $1$ in $d_0$ of the
multiplicands in \eqref{e:Bprod} and choosing a summand $(a_{i,k}^p-a_{i,k})u^{pj-j+1}$ in $d_j$
of the multiplicands for each $1 \le j \le r$, we obtain a term which contributes to the coefficient of $u^{-rp}$.  The contribution
from all such terms will be a multiple of $e_{\sum_{j \ge 1} d_j}(a_{k,1}^p-a_{k,1},\dots,a_{k,p_k}^p-a_{k,p_k})$ and a straightforward counting argument shows that the coefficient on $e_{\sum_{j \ge 1} d_j}(a_{i,1}^p-a_{i,1},\dots,a_{i,p_i}^p-a_{i,p_i})$ which arises from $\bd \in \bD_{i,r}$ is
$$
\frac{\binom{p_i}{d_0, d_1, \dots, d_s}}{\binom{p_i}{\sum_{j \ge 1} d_j}} = \frac{(\sum_{j \ge 1} d_j)!}{\prod_{j \ge 1} d_j!}.
$$
We deduce that each $\bd \in \bD_{i,r}$ contributes 
\begin{equation} \label{e:Bdcoeff}
\frac{(\sum_{j \ge 1} d_j)!}{\prod_{j \ge 1} d_j!} e_{\sum_{j \ge 1} d_j}(a_{i,1}^p-a_{i,1},\dots,a_{i,p_i}^p-a_{i,p_i}).
\end{equation}
to the coefficient of $u^{-pr}$ in \eqref{e:Bprod}.

We note that our definition of $s$ is chosen precisely so that all sequences
$\bd = (d_0,d_1,d_2,\dots)$ of non-negative integers such that
$rp = \sum_{j \ge 1} d_j (jp-j+1)$, have $d_i = 0$ for $i > s$.
So the considerations above give all coefficients of $u^{-rp}$.
Therefore, the coefficient of $u^{-rp}$ in \eqref{e:Bprod} is the sum over all $\bd \in \bD_{i,r}$ of the terms
given in \eqref{e:Bdcoeff}, which proves the first claim of the lemma.

Now we observe that $(p_i-r, r, 0,\dots, 0) \in \bD_{i,r}$ is the unique element which maximises
$\sum_{j \ge 1} d_j$. It is easily verified that $(p_i-r, r, 0,\dots, 0) \in \bD_{i,r}$.
To see that $\sum_{j\ge 1} d_j$ is maximised we observe that for $\bd \in \bD_{i,r}$ we have
$pr = \sum_{j \ge 1} d_j (j(p-1) + 1) \ge p\sum_{j \ge 1} d_j$. We now show that this is
the unique element of $\bD_{i,r}$ with $\sum_{j \ge 1} d_j = r$. Let $\bd \in \bD_{i,r}$.
From the equation $\sum_{j\ge 1} d_j(j-1) = p\sum_{j\ge 1} d_j j - pr$ we deduce
that $p$ is a factor of $\sum_{j\ge 1} d_j(j-1)$, say $mp = \sum_{j\ge 1} d_j(j-1) = \sum_{j\ge 2} d_j(j-1)$.
Substituting back into $rp = \sum_{j \ge 1} d_j (jp-j+1)$ we have
\begin{eqnarray}
\label{e:kpequation}
mp &=& p \sum_{j \ge 1} j d_j - pr = p\left(\sum_{j \ge 1} (j-1) d_j + \sum_{j\ge 1} d_j - r\right) = p\left(mp + \sum_{j\ge 1} d_j - r\right)
\end{eqnarray}
Finally we arrive at $r = m(p-1) + \sum_{j\ge 1} d_j$, and we conclude that
if $r = \sum_{j\ge 1} d_j$ then $m = 0$, which forces $d_2 = d_3 = \cdots = d_s = 0$.
Using $\sum_{j \ge 0} d_j = p_i$ we deduce that
$\bd = (p_i - r, r, 0,\dots,0)$. We have now proven that claim
that $(p_i - r, r, 0, \dots, 0)$ uniquely maximises $\sum_{j\ge 1} d_j$ in $\bD_{i,r}$.

Since $(\sum_{j\ge 1} d_j)! / (\prod_{j \ge 1} d_j!) = 1$ for $\bd = (p_i - r, r, 0,\dots,0)$ it follows that for $i$ fixed
there is a upper unitriangular matrix $C = (c_{s,r})_{1 \le s,r \le p_i}$ such that
$$
\dB_i^{(rp)}v_{A,+} = \sum_{s\leq r} c_{s,r} e_s(a_{i,1}^p - a_{i,1},\dots, a_{i,p_i}^p - a_{i,p_i}).
$$
If we take $C^{-1} = (\tc_{s,r})_{1\leq s,r \leq n}$ and define $\hat B_i^{(rp)} = \sum_{s\leq r} \tc_{s,r} \dB_i^{(ps)}$
then the elements $\hat B_i^{(rp)}$ will act on $v_{A,+}$ in the manner claimed in the lemma.

To finish the proof, we are left to show that if $p > r$, then $\hat B_i^{(rp)} = \dB_i^{(rp)}$,
which will follow from showing that $\bD_{i,r} = \{(p_i - r, r, 0,\dots,0)\}$ under the assumption that $p > r$.
So suppose that $p > r$ and let $\bd \in \bD_{i,r}$. From equation \eqref{e:kpequation} we have
\begin{equation}
\label{e:givesacontra}
p\left(\sum_{j\ge 1}d_j + mp\right) = rp + mp
\end{equation}
Since $\sum_{j\ge 1}d_j  > 0$ we have $p(\sum_{j\ge 1}d_j + mp) > mp^2 = m(p-1)p + mp$. If $m > 0$,
then the hypothesis $p > r$ implies that $m(p-1) \ge r$ and combining with the previous inequality we
arrive at $p(\sum_{j\ge 1}d_j + mp) > rp + mp$, which contradicts \eqref{e:givesacontra}.
We conclude that $m = 0$ and, following the observations made after \eqref{e:kpequation},
we deduce that $\bd = (p_i - r, r, 0, \dots, 0)$. This completes the proof.
\end{proof}

Our next corollary implies that certain elements of $Z(\Ynls)$ are determined by their action on highest
weight vectors.  We need to set up some notation for its statement and proof.

Let $\Ynls_0$ be the subalgebra of $\Ynls$ of all elements fixed by the action of $T_n$.
The PBW basis \eqref{e:YnlsPBWbasis} is $T_n$-stable, and $\Ynls_0$ has a basis consisting of those monomials such
that $\sum_{(i,j,r) \in \bJ_F} u_{i,j}^{(r)} (\epsilon_i - \epsilon_j) = \sum_{(i,j,r) \in \bJ_E} v_{i,j}^{(r)} (\epsilon_i - \epsilon_j)$.
The subspace $\Ynls_{0,\sharp}$ of $\Ynls_0$ spanned by monomials with $\bu \neq 0$ is equal to the
subspace spanned by monomials with $\bv \neq 0$, and thus this subspace is an ideal.
Further, we have a direct sum decomposition $\Ynls_0 = \Ynls^0 \oplus \Ynls_{0,\sharp}$.
We define
$$
\zeta : \Ynls_0 \to \Ynls^0,
$$
to be the projection along this direct sum decomposition.

Recall the basis for $Z(\Ynls)$ given in \eqref{e:basisforthecentre}, and define $Z(\Ynls)^0$ to be the subspace of
of $Z(\Ynls)$ spanned by the monomials with $\bu = \bv = 0$. Clearly
$Z(\Ynls)^0 \subseteq \Ynls_0$.
We write $Z_p(\Ynls)^0$ for the subalgebra of $Z(\Ynls)$ which is
generated by $\{\hat B_i^{(rp)} \mid 1 \leq i \leq n, 0< r \leq p_i\}$; it is a polynomial
algebra on these generators thanks to Theorem~\ref{T:centreY}(b) and Lemma~\ref{L:Birpaction}. We
note that $Z(\Ynls)^0$ is not a subalgebra of $Z(\Ynls)$ %because $\dZ_i^p$ is not an element of $Z(\Ynls)^0$ in general.
but nonetheless, $Z(\Ynls)^0$ is a free $Z_p(\Ynls)^0$-module with basis given by the
restricted monomials given in \eqref{e:freebasiscentreTSY}.

\begin{Corollary} \label{C:ZY0inj}
$ $
\begin{enumerate}
\item[(a)] The restriction of $\zeta$ to $Z(\Ynls)^0$ is injective.
\item[(b)] Let $z \in Z(\Ynls)^0$.  Then $zv_{A,+} = 0$ for all
$A \in \Tab(\pi)$ if and only if $z = 0$.
\end{enumerate}
\end{Corollary}

\begin{proof}
Thanks to Corollary~\ref{C:basisforthecentre} and Lemma~\ref{L:Birpaction} we know that
$Z(\Ynls)^0$ has a basis consisting of ordered monomials
\begin{equation} \label{e:basisY0}
\{\hat  B^\bt \dZ^\bw \mid \bt \in \bI_D, \bw \in \{0,\dots,p-1\}^N\}.
\end{equation}
Let $R = \kk[x_{i,j} \mid 1\leq i \leq n,\, 0 < j \leq p_i]$ be the polynomial ring in variables $x_{i,j}$.
We define a linear map $\omega : Z(\Ynls)^0 \to R$ by setting
\begin{align*}
\label{e:BgoestoR}
\omega(\dZ_r) &= e_r(x_{i,j} \mid 1\leq i \leq n, 0 < j \leq p_i)\\
\omega(\hat B_i^{(rp)}) &=  e_{r}(x_{i,1}^p-x_{i,1},\dots,x_{i,p_i}^p-x_{i,p_i})
\end{align*}
and then extending multiplicatively.

Thanks to \eqref{e:Zact} and Lemma~\ref{L:Birpaction} we know that the action of any element of $Z(\Ynls)^0$ on the Verma module $M(\bar A)$ is given by the composition $p_A \circ \omega$ where $p_A : R \to \kk$ is the homomorphism determined by $x_{i,j} \mapsto a_{i,j}$. In other words, for $z \in Z(\Ynls)^0$ and $A \in \Tab(\pi)$ we have $z v_{A,+} = (p_A \circ \omega(z)) v_{A, +}$. Since we have $z v_{A,+} = \zeta(z) v_{A,+}$ for every $z \in \Ynls_0$, and $\bigcap_{A \in \Tab(\pi)} \ker p_A = 0$, we conclude by Lemma~\ref{L:Dacting} that $\ker \zeta|_{Z(\Ynls)^0} = \ker \omega$. The rest of the proof is
devoted to showing that $\ker \omega = 0$, which implies both (a) and (b).

Let $S := \omega(Z(\Ynls)^0) \sub R$ and $S_p := \omega(Z_p(\Ynls)^0)$.
In order to show that $\omega$ is injective we show that it sends the basis of
$Z(\Ynls)^0$ given in \eqref{e:basisY0} to a basis of $S$.  To this end we show that $S_p$ is a
polynomial ring generated by $\{\omega(\hat B_i^{(rp)}) \mid 1\leq i \leq n, 0 < r \leq p_i\}$,
and that an $S_p$-basis is given by $\{\omega(\dZ^\bw) \mid \bw \in \{0,1,\dots,p-1\}^N\}$.

We place a filtration on $R$ with every $x_{i,j}$ in degree 1, and we have induced filtrations on $S$ and $S_p$.
We identify the associated graded space of $S$ with a subspace of $R$
and we see that $\gr_r e_r(x_{i,j} \mid 1\leq i \leq n, 0 < j \leq p_i) = e_r(x_{i,j} \mid 1\leq i \leq n, 0 < j \leq p_i)$ (as all the
monomials lie in filtered degree $r$), 
whereas $\gr_{pr} e_{r}(x_{i,1}^p-x_{i,1},\dots,x_{i,p_i}^p-x_{i,p_i}) = e_{r}(x_{i,1}^p,\dots,x_{i,p_i}^p) = e_{r}(x_{i,1},\dots,x_{i,p_i})^p$;
in particular, we observe that $\gr S$ is in fact a subalgebra of $R$.
Using Lemma~\ref{L:filtfree}
it suffices to show that the
$p$-restricted monomials in $\{e_r(x_{i,j} \mid 1\leq i \leq n, 0 < j \leq p_i) \mid r = 1,\dots,N\}$ form a basis for $\gr S$ over $\gr S_p$.

At this stage in the proof, we restrict to the case where $\bp = (1^N)$,
because the other cases follow from this case, whilst the notation in this case is more transparent.
Since $n = N$ and $p_1=\cdots = p_n = 1$ we use the notation $x_i$ instead of $x_{i,1}$ for $i=1,\dots,N$.
and write $e_1,\dots,e_N$ for the elementary symmetric polynomials in $x_1,\dots,x_N$.
The subalgebra $\gr S$ of $R$ is generated by
$\{x_i^p \mid i = 1,\dots,N\} \,\cup\, \{e_r \mid r = 1,\dots,N\}$, and the subalgebra
$\gr S_p$ is generated by $\{x_i^p \mid i = 1,\dots,N\}$.
The restricted monomials $e_1^{w_1} \cdots e_N^{w_N}$ with
$\bw \in \{0,\dots,p-1\}^N$ clearly generate $\gr S$ over $R^p$ so it suffices to show that they are linearly independent.
In turn it is enough to prove that $e_1^{w_1} \cdots e_N^{w_N}$
are linearly independent over the fraction field of $R^p$.

To achieve this we apply some field theory that can be found in \cite[Chapter V]{Bo}.
We write $\K = \kk(x_1,\dots,x_N)$ for the fraction field of $R$, and note that
the fraction field of $R^p$ is $\K^p$.
Next we observe that $\{e_1,\dots,e_N\}$ form a separating transcendence basis
of $\K$ over $\kk$ in the sense of \cite[Definition V.16.7.1]{Bo}.
Therefore, by \cite[Theorem V.16.7.5]{Bo}, we have that
$\{de_1,\dots,de_N\}$ form a $\K$-basis of the space $\Omega_\kk(\K)$ of $\kk$-derivations
of $\K$.  Since any $D \in \Omega_\kk(\K)$ annihilates $\K^p$, we have that $\Omega_{\K^p}(\K) = \Omega_\kk(\K)$,
so that $\{de_1,\dots,de_N\}$ is a $\K$-basis of $\Omega_{\K^p}(\K)$.  Then we can apply
\cite[Theorem V.13.2.1]{Bo} to deduce that
$\{e_1,\dots,e_N\}$ is a $p$-basis of $\K$ over $\K^p$, in the sense of
\cite[Definition V.13.1.1]{Bo}.   By definition of a $p$-basis we have
that the $p$-restricted monomials in $\{e_1,\dots,e_N\}$ are a basis of
$\K$ over $\K^p$, and thus in particular are linear independent as required.
\end{proof}

\subsection{Highest weight modules for $U(\g,e)$}
Through the isomorphism $\phi : \Ynls \to U(\g,e)$, which
we know is $T_n$-equivariant, we have a notion of highest weight modules
for $U(\g,e)$.  We use the notation and terminology introduced in \S\ref{ss:hwY}
also for $U(\g,e)$.
We are mainly interested in considering the restriction of highest weight
$U(\h)$-modules to $U(\g,e)$, and our main result is Lemma~\ref{L:DonUh}.
We move on to show that elements of $Z(\g,e)^0$ are determined by their action
on highest weight vectors in Corollary \ref{C:ZU0inj}.

We recall the good grading $\g = \bigoplus_{i \in \Z} \g(i)$ from \eqref{e:goodgrading}
and the notation $\h :=  \g(0)$ and $\p = \bigoplus_{i \ge 0} \g(i)$ from \eqref{e:phandm}.
We recall that the heights of the columns in $\pi$ are $q_1,\dots,q_l$, and so
$\h \cong \gl_{q_1}(\kk)  \oplus \cdots \oplus \gl_{q_l}(\kk)$.
We let $\b_\h$ be the Borel subalgebra of $\h$ with basis $\{e_{i,j} \mid \col(i)=\col(j),\, \row(i) \le \row(j) \}$,
which is the direct sum of the Borel subalgebras of upper triangular matrices in each of the $\gl_{q_i}(\kk)$.

For $A \in \Tab_\kk(\pi)$
we define the weight $\lambda_A \in \t^*$ by
\begin{equation*}
\lambda_A := \sum_{i=1}^N a_i \epsilon_i.
\end{equation*}
We let
$$
\rho_\h := - \sum_{i=1}^N (\row(i)-1) \epsilon_i,
$$
which is a ``shifted choice of $\rho$ for the Borel subalgebra $\b_\h$ of $\h$''.
Then we define
$$
\widetilde \rho =  \eta + \rho_\h,
$$
where we recall that $\eta$ is defined in \eqref{e:eta}.

We define $\kk_A$ to be the 1-dimensional
$\t$-module on which $\t$ acts via $\lambda_A-\widetilde{\rho}$, and view it
also as a module for $\b_\h$ on which the nilradical acts trivially.
Then we define the Verma module $M_\h(A) = U(\h) \otimes_{U(\b_\h)} \kk_{A}$ for $U(\h)$, and we
write $m_A := 1 \otimes 1_A$ for the highest weight vector.
We may view $M_\h(A)$ as a $U(\p)$-module on which the nilradical $\bigoplus_{i>0} \g(i)$ of $\p$
acts trivially, and then restrict it to $U(\g,e) \sub U(\p)$.  We write $\bar M_\h(A)$
for the restriction of $M_\h(A)$ to $U(\g,e)$, and write $\bar m_A$ for $m_A$ viewed
as an element of $\bar M_\h(A)$.

The following lemma shows that $\bar m_A$ is a highest weight vector in $\bar M_\h(A)$
with weight $\bar A$,
and further gives the action of $\xi_\p(\gr' D_i^{(r)})$ on
$\bar m_A$.
We note that a proof of (b) could be given based on the last paragraph
of the proof of \cite[Theorem~7.9]{BKrep}; however we give a more direct approach here,
which can also be used to prove (c).

\begin{Lemma} \label{L:DonUh}
Let $A \in \Tab_\kk(\pi)$ and let $\bar M_\h(A)$ and $\bar m_A$ be as defined above.
Then
\begin{enumerate}
\item[(a)] $E_{i,j}^{(r)} \bar m_A = 0$ for all $(i,j,r) \in \bJ_E$;
\item[(b)] $D_i^{(r)} \bar m_A = e_r(a_{i,1}+(i-1),\dots,a_{i,p_i}+(i-1))$ for all $(i,r) \in \bJ_D$; and
\item[(c)]
$\xi_\p(\gr' D_i^{(r)}) \bar m_A = e_r(a_{i,1}^p-a_{i,1},\dots,a_{i,p_i}^p-a_{i,p_i})$ for all $(i,r) \in \bJ_D$.
\end{enumerate}
\end{Lemma}

\begin{proof}
First we note that
$M_\h(A)$ is isomorphic as a $U(\p)$-module to $U(\p)/I_\p(A)$, where $I_\p(A)$ is the left ideal of $U(\p)$ generated by
$\{e_{i,j} - \delta_{i,j}(\lambda_A - \widetilde \rho)(e_{i,i}) \mid \col(i)= \col(j),\, \row(i) \le \row(j)\}\cup \{e_{i,j} \mid \col(i) > \col(j)\}$.
Next we observe that $T^e \iso T_n$ acts on $\p$ by the adjoint action, and this induces an action
of $T_n$ on $I_\p(A)$.  Using the same proof as  Lemma~\ref{L:definingVermas}(a) we see that any element of
$U(\p)$ with a positive $T_n$ weight annihilates $m_A$. Now part (a) follows as $E_i^{(r)} \in U(\g,e) \sub U(\p)$ has positive $T_n$-weight.

We move on to prove (b), where we use the explicit formula for $D_i^{(r)}$ given in \eqref{e:Dir}.
We set up some notation to simplify the proof. The formula \eqref{e:Dir} is given as a sum of terms
indexed by integers $1 \leq i_1,...,i_s, j_1,...,j_s \leq N$ subject to conditions (a)--(f).
We write $\bi = (i_1,...,i_s), \bj = (j_1,...,j_s)$ and $\te_{\bi,\bj}$ for the summand corresponding to $\bi, \bj$.

First we observe that if $s < r$, then condition (a) ensures that $\col(j_k) > \col(i_k)$ for some $k$, which
implies that $\te_{\bi, \bj}$ kills $m_A$.

Now we consider sequences $\bi, \bj$ with $s=r$.  Then we have $\col(i_k) = \col(j_k)$ for all $k$,
so that $\te_{\bi,\bj} \in U(\h)$.  Suppose that $i_k \ne j_k$ for all $k$.
Using conditions (d), (e) and (f) we see that there is some $k$ such that $i_k < j_k$,
and we choose the maximal such $k$.
We certainly have that $\te_{i_k,j_k} = e_{i_k,j_k}$ kills $m_A$.  Further by
condition (c) and (e), we have $\col(i_{m}) > \col(i_k)$ for all $m > k$, so
that $e_{i_k,j_k}$ commutes with $\te_{i_m,j_m}$.  We deduce
$\te_{\bi,\bj}$ kills $m_A$.

Hence, we see that the only summands $\te_{\bi, \bj}$ in $D_i^{(r)}$ which do not
kill $m_A$ correspond to sequences $\bi = \bj = (i_1,...,i_r)$, where $\row(i_k) = i$ for all $k$, and $i_1 < i_2 < \cdots < i_r$.
We have $(\lambda_A - \tilde\rho)(\te_{k,k}) = (\lambda_A - \rho_\h)(e_{k,k}) = a_k + (\row(k) - 1)$ for $k=1,...,N$ and it follows that
$D_i^{(r)}$ acts on $\bar m_A$ by
$$
\sum_{\substack{i_1 < \cdots < i_r\\ \row(i_k) = i}}
(a_{i_1} + (i - 1))\cdots (a_{i_r} + (i - 1)) = e_r(a_1 + (i-1), \dots ,a_{p_i} - (i-1)).
$$

To prove (c), we can argue exactly as above and use the formula for $\xi_\p(D_i^{(r)})$
given in \eqref{e:explicitDpcentre}; in fact the argument is easier as the monomials in the 
expression for $\xi_\p(D_i^{(r)})$ consist of commuting terms.  This shows that $\xi_\p(\gr' D_i^{(r)})$ acts on
$\bar m_A$ via
$$
(\lambda_A-\tilde \rho)
\left( \sum_{\substack{i_1 < \cdots < i_r\\ \row(i_k) = i}}   (e_{i_1, i_1}^p - e_{i_1, i_1}) \cdots (e_{i_r, i_r}^p - e_{i_r, i_r})\right).
$$
We have that $\lambda_A(e_{i_k, i_k}^p - e_{i_k, i_k}) = a_{i,k}^p-a_{i,k}$ whilst
$\tilde \rho(e_{i_k, i_k}^p - e_{i_k, i_k}) = \tilde \rho(e_{i_k, i_k})^p - \tilde \rho(e_{i_k, i_k}) = 0$.
Hence, $\xi_\p(\gr' D_i^{(r)})$ acts on $\bar m_A$ via $e_r(a_{i,1}^p-a_{i,1},\dots,a_{i,p_i}^p-a_{i,p_i})$ as required.
\end{proof}

To end the subsection, we record a version of Corollary~\ref{C:ZY0inj}(b) for the algebra $U(\g,e)$. 
We define $Z(\g,e)^0$ to be the subspace
of $Z(\g,e)$ which is spanned by the PBW monomials appearing in \eqref{e:Zgebasis} such that $\bu = \bv = 0$.

\begin{Corollary} \label{C:ZU0inj}
Let $z \in Z(\g,e)^0$.  Then $zv_{A,+} = 0$ for all
$A \in \Tab(\pi)$ if and only if $z = 0$.
\end{Corollary}

\begin{proof}
It follows from Lemma \ref{L:HClemmaUg2}, along with \eqref{e:Zact}, that $Z_r \in Z_{\HC}(\g,e)$ acts on $\bar m_A$ via the $r$th elementary symmetric function in $\{a_{i,j} \mid 1\leq i\leq n, 1\leq j \leq p_i\}$.
Also by Lemma~\ref{L:DonUh}(c), we know that $\xi_\p(\gr' D_i^{(r)}) \in Z_p(\g,e)$ acts on $\bar m_A$ by the $r$th elementary symmetric polynomial in
$\{a_{i,1}^p - a_{i,1},\dots,a_{i,p_i}^p - a_{i,p_i}\}$.
Therefore, we may apply precisely the same argument as for Corollary~\ref{C:ZY0inj}(b) to complete the proof.
\end{proof}

\section{The isomorphism of restricted versions}
\label{S:mainproof}

The main goal of this section is to prove Theorem~\ref{T:restricted}.  We continue to use
that notation introduced in Sections \ref{S:YandW}--\ref{S:hwY}.

\begin{Lemma} \label{L:EsandFs}
$ $
\begin{enumerate}
\item[(a)] $\phi((\dE_{i,j}^{(r)})^p)\in Z_p(\g,e)_+Z(\g,e)$.
\item[(b)]  $\phi((\dF_{i,j}^{(s)})^p) \in Z_p(\g,e)_+Z(\g,e)$.
\item[(c)] $\phi(\hat B_i^{(rp)}) - \xi_\p(\gr' D_i^{(r)}) \in Z_p(\g,e)_+Z(\g,e)$.
\end{enumerate}
\end{Lemma}

\begin{proof}
Recall from \textsection\ref{ss:restrictedW} that $\xi_\p(\gr' E_{i,j}^{(r)}{}^p), \xi_\p(\gr' F_{i,j}^{(r)}{}^p) \in Z_p(\g,e)_+$.
The basis elements of $Z(\g,e)$ in \eqref{e:Zgebasis} with nonzero weight
have $\bu \ne 0$ or $\bv \ne 0$, so that these elements lie in $Z_p(\g,e)_+Z(\g,e)$.
Now (a) and (b) follow from the facts that $\dE_{i,j}^{(r)}$ has $T_n$-weight
$p(\epsilon_i - \epsilon_j)$ and $\dF_{i,j}^{(r)}$ has $T_n$-weight $p(\epsilon_j - \epsilon_i)$
along with the fact that $\phi$ is $T_n$-equivariant.

By Lemmas~\ref{L:Birpaction} and \ref{L:DonUh} we know that $\phi(\hat B_i^{(rp)}) - \xi_\p(\gr' D_i^{(r)})$ acts trivially on all highest weight vectors $\bar m_A$ for $U(\g,e)$. It follows from the definition of $B_i(u)$ that $\dB_i^{(rp)}$
is fixed by $T_n$, thus $\phi(\hat B_i^{(rp)})$
is also fixed.
Similarly, $\xi_\p(\gr' D_i^{(r)})$ is centralized by $T_n$, because $D_i^{(r)}$ is, and $\xi_\p$ is $T_n$-equivariant.
Therefore, $\phi(\hat B_i^{(rp)}) - \xi_\p(\gr' D_i^{(r)})$.
Now writing $\phi(\hat B_i^{(rp)}) - \xi_\p(\gr' D_i^{(r)})$ as a sum of the basis elements of $Z(\g,e)$ given in \eqref{e:Zgebasis} we
deduce that $\phi(\hat B_i^{(rp)}) - \xi_\p(\gr' D_i^{(r)})$ is a span of elements with
$\bu = \bv = 0$ modulo terms lying in $Z_p(\g,e)_+Z(\g,e)$.
We may now apply Corollary~\ref{C:ZU0inj} to deduce
that $\phi(\tB_i^{(rp)}) - \xi_\p(\gr' D_i^{(r)}) \in Z_p(\g,e)_+Z(\g,e)$ as required.
\end{proof}

We are now ready to deduce our main theorem.

\begin{proof}[Proof of Theorem \ref{T:restricted}]
Using Lemma~\ref{L:EsandFs} along with the fact that $\xi_\p(\gr' D_i^{(r)}) \in Z_p(\g,e)_+$ we know that $\phi$ maps $Z_p(\Ynls)_+$ to $Z_p(\g,e)_+Z(\g,e)$, and it follows immediately that
$\phi(\Ynls Z_p(\Ynls)_+) \subseteq U(\g,e)Z_p(\g,e)_+$.  We conclude that $\phi$ induces a surjective map
$\phi^{[p]} : \Ynlps \to U^{[p]}(\g,e)$.  Moreover, $\dim \Ynlps = p^{\dim \g^e} = \dim U^{[p]}(\g,e)$
by considering the bases given in \eqref{e:PBWrestY} and \eqref{e:PBWrestW}.  Hence $\phi^{[p]}$ is an isomorphism.
\end{proof}

\end{document}